\documentclass[runningheads]{llncs}
\usepackage{bussproofs}
\EnableBpAbbreviations
\usepackage{amsmath,amssymb}
\usepackage{mathtools}
\usepackage{txfonts}
\usepackage{verbatim,cmll}
\usepackage{comment}
\usepackage{mathdots}
\usepackage{multicol}
\usepackage{xspace}
\usepackage{color}
\usepackage{xcolor}
\usepackage{graphicx}
\usepackage{cmll}
\usepackage{appendix}
\usepackage{multirow}

\def\mc{\multicolumn}

\def\fCenter{{\mbox{$\ \vdash\ $}}}
\newcommand{\fns}{\footnotesize}

\newcommand{\osigma}{\overline{\sigma}}
\renewcommand{\emph}{\textbf}

\newcommand{\Prop}{\mathsf{Prop}}

\newcommand{\NOMJ}{\mathbf{J}}
\newcommand{\NOMH}{\mathbf{H}}

\newcommand{\nomi}{\mathbf{i}}
\newcommand{\nomj}{\mathbf{j}}
\newcommand{\nomh}{\mathbf{h}}
\newcommand{\nomk}{\mathbf{k}}

\newcommand{\NCT}{\mathbf{T}}

\newcommand{\CNOMM}{\mathbf{M}}
\newcommand{\CNOMN}{\mathbf{N}}

\newcommand{\cnomm}{\mathbf{m}}
\newcommand{\cnomn}{\mathbf{n}}
\newcommand{\cnomo}{\mathbf{o}}

\renewcommand{\phi}{\varphi}

\newcommand{\wbox}{\ensuremath{\Box}\xspace}
\newcommand{\wdia}{\ensuremath{\Diamond}\xspace}
\newcommand{\bbox}{\ensuremath{\blacksquare}\xspace}
\newcommand{\bdia}{\ensuremath{\Diamondblack}\xspace}

\newcommand{\WBOX}{\ensuremath{\check{\Box}}\xspace}
\newcommand{\WDIA}{\ensuremath{\hat{\Diamond}}\xspace}
\newcommand{\BBOX}{\ensuremath{\check{\blacksquare}}\xspace}
\newcommand{\BDIA}{\ensuremath{\hat{\Diamondblack}}\xspace}

\newcommand{\aatop}{\ensuremath{\top}\xspace}
\newcommand{\abot}{\ensuremath{\bot}\xspace}
\newcommand{\aand}{\ensuremath{\wedge}\xspace}
\newcommand{\aor}{\ensuremath{\vee}\xspace}

\newcommand{\ATOP}{\hat{\top}}
\newcommand{\AATOP}{\hat{\top}}
\newcommand{\ABOT}{\ensuremath{\check{\bot}}\xspace}

\def\fCenter{\vdash}

\newcommand{\vd}{\ \,\textcolor{black}{\vdash}\ \,}

\newcommand{\ol}[1]{\overline{#1}}

\newcommand{\Rarr}{\Rightarrow}

\title{Labelled calculi for lattice-based modal logics}
\author{
Ineke van der Berg\inst{1,3}\orcidID{0000-0003-2220-1383}
\and
Andrea De Domenico\inst{1}\orcidID{0000-0002-8973-7011} 
\and
Giuseppe Greco\inst{1}\orcidID{0000-0002-4845-3821} 
\and
Krishna B.~Manoorkar\inst{1}\orcidID{0000-0003-3664-7757}
\and
Alessandra Palmigiano\inst{1,2}\orcidID{0000-0001-9656-7527}
\and
Mattia Panettiere\inst{1}\orcidID{0000-0002-9218-5449}
}
\authorrunning{van der Berg, De Domenico, Greco, Manoorkar, Palmigiano, Panettiere}
\institute{School of Business and Economics, Vrije Universiteit Amsterdam, the Netherlands 
\and
Department of Mathematics and Applied Mathematics, U.~of Johannesburg, South Africa
\and
Department of Mathematical Sciences, Stellenbosch University
\\
\email{\{i.van.der.berg,a.de.domenico,g.greco,k.b.manoorkar,\\a.palmigiano,m.panettiere\}@vu.nl}
}
\date{September 2022}

\begin{document}

\maketitle
\begin{abstract}
 We introduce  labelled sequent  calculi for the basic normal non-distri-butive modal logic $\mathbf{L}$ and 31 of its axiomatic extensions, where the labels  are atomic formulas of a first order language which is interpreted on the canonical extensions of the algebras in the variety corresponding to the logic $\mathbf{L}$. Modular proofs are presented that these calculi are all sound, complete and conservative w.r.t.~$\mathbf{L}$, and enjoy cut elimination and the subformula property. The introduction of these calculi showcases a  general methodology for introducing labelled calculi for the  class of LE-logics and their analytic axiomatic extensions in a principled and uniform way. 
\end{abstract}
\section{Introduction}
The present paper pertains to a line of research in structural proof theory aimed at generating analytic calculi for wide classes of nonclassical logics in a principled and uniform way. Since the 1990s, semantic information about given logical frameworks has proven key to generate calculi with excellent properties \cite{Sim94}. The contribution of semantic information has been particularly perspicuous in the introduction of labelled calculi for e.g.~classical normal modal logic  \cite{negri2005proof} and intuitionistic logic \cite{Negri2012}, and their  axiomatic extensions defined by axioms for which first-order correspondents exist of a certain syntactic shape \cite{MASSA}. Moreover, recently, the underlying link between the principled and algorithmic generation of analytic rules for capturing axiomatic extensions of given logics and the systematic access to, and use of,  semantic information for this purpose has been established also in the context of other proof-theoretic formats, such as proper display calculi \cite{GMPTZ,ChnGrePalTzi21}, and relative to classes of logics as wide  as the normal (D)LE-logics, i.e.~those logics canonically associated with varieties of normal (distributive) lattice expansions \cite{CP-nondist} (cf.~Definition 1.1). In particular, in \cite{GMPTZ},  the same algorithm ALBA which computes the first-order correspondents of (analytic) inductive axioms in any (D)LE-signature was used to generate the analytic rules in a suitable proper display calculus corresponding to those axioms.

The algorithm ALBA \cite{CoPa12,CP-nondist} is among the main tools in unified correspondence theory \cite{CoGhPa14}, and allows not only for the mechanization  of well known correspondence arguments  from modal logic, but also for the uniform generalization of these arguments to (D)LE-logics, thanks to the fact that  the ALBA-computations are motivated by and interpreted in an algebraic environment in which the classic model-theoretic correspondence arguments can be rephrased in terms of the order-theoretic properties of the algebraic interpretations of the logical connectives. These properties guarantee the soundness of the rewriting rules applied in ALBA-computations, thanks to which, the first-order correspondent of a given input axiom (in any given LE-language $\mathcal{L}$) is generated in a language $\mathcal{L}^+$ expanding $\mathcal{L}$, which is interpreted in the canonical extensions of  $\mathcal{L}$-algebras.

In the present paper, we showcase how the methodology  adopted in \cite{GMPTZ} for introducing proper display calculi for (D)LE-logics and their analytic axiomatic extensions can be used also for endowing LE-logics with labelled calculi. Specifically, we focus on a particularly simple LE-logic, namely the basic normal non-distributive (i.e.~lattice-based) modal logic $\mathbf{L}$ \cite{conradie2017toward,conradie2016categories}, for which we  introduce a labelled calculus and show its basic properties, namely soundness, completeness, cut-elimination and subformula property. Moreover, we discuss, by way of examples, how ALBA can be used to generate analytic rules corresponding to (analytic inductive) axiomatic extensions of the basic logic $\mathbf{L}$.

\paragraph{Structure of the paper.} Section \ref{sec:preliminaries} recalls preliminaries on basic normal non-distributive logic, canonical extensions and the algorithm ALBA, Section \ref{sec:LabelledCalculusAL} presents a labelled calculus for normal non-distributive logic and its extensions. Section \ref{sec:PropertiesOfTheValculusAL} proves soundness, completeness, cut elimination and subformula property for basic normal non-distributive logic and some of its axiomatic extensions. Section \ref{ALIsAProperLabelledCalculus} shows that the all calculi introduced in the paper are proper labelled calculi. We conclude in Section \ref{sec:Conclusions}. In Appendix \ref{sec:ProperAlgebraicLabelledCalculi} we provide the formal definition of proper labelled calculi and we show that any calculus in this class enjoys the canonical cut elimination {\em \`a la} Belnap.

\section{Preliminaries}
\label{sec:preliminaries}

\subsection{Basic normal non-distributive modal logic, its associated ALBA-language, and some of its axiomatic extensions}
\label{ssec:Non-distributive modal logic}
The basic normal non-distributive modal logic is  a normal LE-logic (cf.~\cite{CP-nondist,conradie2020non}) which was used in \cite{conradie2016categories,conradie2017toward} as the underlying environment for an epistemic logic of categories and formal concepts, and in \cite{conradie2021rough} as the underlying environment of a logical theory unifying Formal Concept Analysis \cite{ganter2012formal} and Rough Set Theory \cite{pawlak1982rough}. 

Let $\Prop$ be a (countable or finite) set of atomic propositions. The language $\mathcal{L}$ is defined as follows:
\[
  \varphi \coloneqq \bot \mid \top \mid p \mid  \varphi \wedge \varphi \mid \varphi \vee \varphi \mid \Box \varphi \mid  \Diamond\varphi,  
\]

\noindent where $p\in \Prop$. 
The extended language $\mathcal{L}^+$, used in ALBA-computations taking inequalities of $\mathcal{L}$-terms in input, is defined as follows:
\[
  \psi \coloneqq \nomj\mid \cnomm \mid \varphi  \mid  \psi \wedge \psi \mid \psi \vee \psi \mid \wbox \psi \mid  \wdia\psi \mid \bbox \psi \mid  \bdia\psi,  
\]
where $\varphi\in \mathcal{L}$, and the variables $\nomj\in \mathsf{NOM}$ (resp.~$\cnomm\in \mathsf{CNOM}$), referred to as {\em nominals} (resp.~{\em co-nominals}), range over disjoint sets which are also disjoint from $\Prop$.
The {\em basic}, or {\em minimal normal} $\mathcal{L}$-{\em logic} is a set $\mathbf{L}$ of sequents $\phi\vdash\psi$,  with $\phi,\psi\in\mathcal{L}$, containing the following axioms:

{{\centering
\begin{tabular}{ccccccccccccc}
     $p \vdash p$ & \quad\quad & $\bot \vdash p$ & \quad\quad & $p \vdash p \vee q$ & \quad\quad & $p \wedge q \vdash p$ & \quad\quad & $\top \vdash \Box\top$ & \quad\quad & $\Box p \wedge \Box q \vdash \Box(p \wedge q)$
     \\
     & \quad & $p \vdash \top$ & \quad & $q \vdash p \vee q$ & \quad & $p \wedge q \vdash q$ &\quad &  $\Diamond\bot \vdash \bot$ & \quad & $\Diamond(p \vee q) \vdash \Diamond p \vee \Diamond q$
\end{tabular}
\par}}
\noindent and closed under the following inference rules:
		{\small{
		\begin{gather*}
			\frac{\phi\vdash \chi\quad \chi\vdash \psi}{\phi\vdash \psi}
			\quad
			\frac{\phi\vdash \psi}{\phi\left(\chi/p\right)\vdash\psi\left(\chi/p\right)}
			\quad
			\frac{\chi\vdash\phi\quad \chi\vdash\psi}{\chi\vdash \phi\wedge\psi}
			\quad
			\frac{\phi\vdash\chi\quad \psi\vdash\chi}{\phi\vee\psi\vdash\chi}
			\quad
			\frac{\phi\vdash\psi}{\Box \phi\vdash \Box \psi}
\quad
\frac{\phi\vdash\psi}{\Diamond \phi\vdash \Diamond \psi}
		\end{gather*}
		}}
An {\em $\mathcal{L}$-logic} is any extension of $\mathbf{L}$  with $\mathcal{L}$-axioms $\phi\vdash\psi$. In what follows, for any set $\Sigma$ of $\mathcal{L}$-axioms, we let $\mathbf{L}.\Sigma$ denote the axiomatic extension of $\mathbf{L}$ generated by $\Sigma$. Throughout the paper, we will consider all subsets $\Sigma$ of the set of axioms listed in the table below. These axioms are well known from classical  modal logic, and have also cropped up in  \cite{conradie2021rough} in the context of the definition of  relational structures  simultaneously generalizing Formal Concept Analysis and Rough Set Theory.
{\small
\begin{center}
\begin{tabular}{rccclclcccl}
\hline
(4)  && $\wdia \wdia A \vdash \wdia A$ && {\fns transitivity} &$\quad\quad$& (D) && $\wbox A \vdash \wdia A$ && {\fns seriality} \\
(T)  && $\wbox A \vdash A$ && {\fns reflexivity} && (C)  && $\wdia \wbox A \vdash \wbox \wdia A$ && {\fns confluence} \\
(B)  && $A \vdash \wbox \wdia A$ && {\fns symmetry} && \\

\hline
\end{tabular}
\end{center}
 }

\subsection{$\mathcal{L}$-algebras, their canonical extensions, and the algebraic interpretation of the extended language of ALBA}
\label{ssec:interpretation ALBA}
In the present section, we recall the definitions of the  normal lattice expansions canonically associated with the basic logic $\mathbf{L}$, their canonical extensions, the existence of which can be shown both constructively and non-constructively, and the interpretation of the extended language $\mathcal{L}^+$ in the canonical extensions of $\mathcal{L}$-algebras.

An {\em $\mathcal{L}$-algebra}  is a tuple $\mathbb{A} = (L,\wdia^\mathbb{A},\wbox^\mathbb{A})$, where $L$ is a bounded lattice, $\wdia^\mathbb{A}$ (resp.~$\wbox^\mathbb{A}$) is a finitely join-preserving (resp.~finitely meet-preserving) unary operation. That is, besides the usual identities defining general lattices, the following identities hold:
\[\wdia(x\aor y) = \wdia x\aor \wdia y\quad\quad \wdia \abot = \abot\quad\quad \wbox (x\aand y) = \wbox x\aand \wbox y\quad\quad \wbox \aatop = \aatop.\]

In what follows, we let $\mathsf{Alg}(\mathcal{L})$ denote the class of $\mathcal{L}$-algebras.
Let $L$ be a (bounded) sublattice of a complete lattice $L'$.
\begin{enumerate}
\item  $L$ is {\em dense} in $L'$ if every element of $L'$ can be expressed both as a join of meets and as a meet of joins of elements from $L$. We let $K(L')$ (resp.~$O(L')$) denote the meet-closure (resp.~join-closure) of $L$ in $L'$. That is, $K(L') = \{k\in L'\mid k = \bigwedge S$ for some $S\subseteq L\}$, and $O(L') = \{o\in L'\mid o = \bigvee T$ for some $T\subseteq L\}$.
\item $L$ is {\em compact} in $L'$ if, for all $S, T \subseteq L$, if $\bigwedge S\leq \bigvee T$ then $\bigwedge S'\leq \bigvee T'$ for some finite $S'\subseteq S$ and $T'\subseteq T$.
\item The {\em canonical extension} of a lattice $L$ is a complete lattice $L^\delta$ containing $L$
as a dense and compact sublattice. Elements in $K(L^\delta)$ (resp.~$O(L^\delta)$) are the {\em closed} (resp.~{\em open}) {\em elements} of $L^\delta$.
\end{enumerate}
As is well known (cf.~\cite{GH01}), the canonical extension of a lattice $L$ exists and is unique up to an isomorphism fixing $L$. The non-constructive proof of existence can be achieved via suitable dualities for lattices, while the constructive proof uses the MacNeille completion construction on a certain poset obtained from the families of proper lattice filters and ideals of the original lattice $L$ (cf.~\cite{GH01,DGP05} for details). In the latter case, the ensuing complete lattice $L^\delta$ can be shown to be {\em perfect}, i.e., to be  both completely join-generated by the set $J^\infty(L^\delta)\subseteq K(L^\delta)$ of the completely
join-irreducible elements of $L^\delta$, and completely meet-generated by the set $M^\infty(L^\delta)\subseteq O(L^\delta)$ of
the completely meet-irreducible elements of $L^\delta$.\footnote{For any complete lattice $L$, any $j\in L$ is completely join-irreducible if $j\neq \bot$ and for any $S\subseteq L$, if $j = \bigvee S$ then $j\in S$. Dually, any $m\in L$ is completely meet-irreducible if $m\neq \top$ and for any $S\subseteq L$, if $m = \bigwedge S$ then $m\in S$.}

For every unary, order-preserving operation $f : L\to L$, the $\sigma$-{\em extension} of $f$ is defined first on any $k\in K(L^\delta)$ and then on every $u\in L^\delta$ as follows:
$$f^\sigma(k):= \bigwedge\{ f(a)\mid a\in L\mbox{ and } k\leq a\} \quad \quad 
f^\sigma(u):= \bigvee\{ f^\sigma(k)\mid k\in K(L^\delta)\mbox{ and } k\leq u\}.$$
The $\pi$-{\em extension} of $f$ is defined first on every $o\in O(L^\delta)$, and then on every $u\in L^\delta$ as follows:
$$f^\pi(o):= \bigvee\{ f(a)\mid a\in L\mbox{ and } a\leq o\}\quad\quad 
f^\pi(u):= \bigwedge\{ f^\pi(o)\mid o\in O(L^\delta)\mbox{ and } u\leq o\}.$$

Defined as above, the $\sigma$- and $\pi$-extensions maps  are monotone, and coincide with $f$ on the elements of $\mathbb{A}$. Moreover,  the $\sigma$-extension (resp.~(resp.~$\pi$-extension) of a finitely join-preserving (resp.~finitely meet-preserving) map is  {\em completely} join-preserving (resp. {\em completely} meet-preserving). This justifies defining
the {\em canonical extension of an
$\mathcal{L}$-algebra} $\mathbb{A} = (L, \wbox, \wdia)$ as the   $\mathcal{L}$-algebra
$\mathbb{A}^\delta: = (L^\delta, \wbox^{\pi}, \wdia^{\sigma})$.
By construction, $\mathbb{A}$ is a subalgebra of $\mathbb{A}^\delta$ for any $\mathbb{A}\in \mathsf{Alg}(\mathcal{L})$. In fact, compared to arbitrary $\mathcal{L}$-algebras, $\mathbb{A}^\delta$ enjoys additional properties that make it a suitable semantic environment for the extended language $\mathcal{L}^+$ of Section \ref{ssec:Non-distributive modal logic}. Indeed, the lattice reduct of $\mathbb{A}^\delta$ is a {\em complete} lattice.  Together with the fact that the operations $\wdia^\sigma$ and $\wbox^\pi$ do not  preserve only {\em finite} joins and meets respectively, but  {\em arbitrary} joins and meets, this implies, by well known order-theoretic facts (cf.~\cite[Proposition 7.34]{LatticesOrder}),  that the right and left adjoint of $\wdia^\sigma$ and of $\wbox^\pi$ are well defined on $\mathbb{A}^\delta$, which we denote $\bbox^{\mathbb{A}^\delta}$ and $\bdia^{\mathbb{A}^\delta}$ respectively,\footnote{The unary operations $\bbox^{\mathbb{A}^\delta}$ and $\bdia^{\mathbb{A}^\delta}$ on $\mathbb{A}^\delta$ are the unique maps satisfying the  equivalences $\wdia^\sigma u\leq v$ iff $u\leq \bbox^{\mathbb{A}^\delta}v$ and $\bdia^{\mathbb{A}^\delta} u\leq v$ iff $u\leq \wbox^\pi v$ for all $u, v\in \mathbb{A}^\delta$. } and provide the interpretations of the corresponding logical connectives in $\mathcal{L}^+$. Moreover, by denseness, $\mathbb{A}^\delta$ is both completely join-generated by the elements in $K(L^\delta)$ and completely meet-generated by the elements in $O(L^\delta)$, and when considering the non-constructive proof, these families of generators can be further restricted to $J^{\infty}(L^\delta)$ and $M^{\infty}(L^\delta)$, respectively. These generating subsets provide the interpretation of the variables in $\mathsf{NOM}$ and $\mathsf{CNOM}$, respectively.
As is well known, for any set $\Sigma$ of $\mathcal{L}$-sequents, if $\mathsf{K}(\Sigma)= \{\mathbb{A}\in \mathsf{Alg}(\mathcal{L}) \mid \mathbb{A}\models \Sigma\}$ is closed under taking canonical extensions,\footnote{By the general theory of unified correspondence (cf.~\cite{CP-nondist}), this is the case of every subset $\Sigma$ of the set of axioms listed at the end of Section \ref{ssec:Non-distributive modal logic}.} then the axiomatic extension  $\mathbf{L}.\Sigma$ is complete w.r.t.~the subclass $\mathsf{K}^\delta (\Sigma)= \{\mathbb{A}^\delta\mid \mathbb{A}\in \mathsf{K}(\Sigma)\}$, because any non-theorem $\xi\vd \chi$ will be falsified in the Lindenbaum-Tarski algebra $\mathbb{A}$ of $\mathbf{L}.\Sigma$, which is an element of  $\mathsf{K}(\Sigma)$ by construction, and hence $\xi\vd \chi$ will be falsified under the same assignment in $\mathbb{A}^\delta$,  given that $\mathbb{A}$ is a subalgebra of $\mathbb{A}^\delta$.

\subsection{The algorithm ALBA}
\label{sec:TheAlgorithmALBA}
The algorithm ALBA is guaranteed to succeed on a large class of formulas, called (analytic) inductive axioms, and it can be used to automatically generate labelled calculi with good properties equivalently capturing the LE-logics axiomatized by means of those axioms. We refer the reader to  \cite[Section 6,8]{CP-nondist} for the proof of correctness and success in the general setting of LE-logics. In the present section, we informally illustrate how the algorithm ALBA works by means of examples, namely, we run ALBA on the modal axioms in $\Sigma =\{\wbox p \vd p, p \vd \wbox \wdia p, \wbox p \vd \wdia p, \wdia \wbox p \vd \wbox \wdia p\}$ computing their first-order correspondent, which, in turn, can be automatically transformed into an analytic structural rule of a labelled calculus equivalently capturing the axioms (see the table at the end of Section \ref{sec:LabelledCalculusAL}). In what follows,  $\mathbb{A}$ denotes an $\mathcal{L}$-algebra, and $\mathbb{A}^\delta$ denotes its canonical extension. We abuse notation and use the same symbol for the algebra and its domain. We recall that variables $\nomj, \nomh$ and $\nomk$ (resp.~$\cnomm$) range in the set of the complete join-generators (resp.~complete meet-generators) of $\mathbb{A}^\delta$.

The following chain of equivalences is sound on $\mathbb{A}^\delta$:
{\small
\vspace{-0.3 cm}
\begin{center}
\begin{tabular}{rcll}
&& $\forall p(\wdia \wdia p \leq  \wdia p)$\\
&  iff & $\forall p\forall \nomj  \forall \cnomm \, ((\nomj\leq p\ \&\ \wdia p\leq \cnomm)\Rarr \wdia\wdia \nomj\leq \cnomm)$  & \ \ \ join- and meet-generation, $\wdia$ c.~join-preserving \\
 & iff & $\forall \nomj  \forall \cnomm \, (\wdia \nomj \leq \cnomm \Rarr \wdia \wdia \nomj  \leq \cnomm)$ & \ \ \ Ackermann's lemma\\
& iff & $\forall \nomj \forall \nomh \forall \cnomm \, (\wdia \nomj \leq \cnomm \Rarr (\nomh \leq \wdia \nomj \Rarr \wdia \nomh \leq \cnomm))$ & \\
\end{tabular}
\end{center}
 }

Indeed, the first equivalence in the chain above is due to the fact that, since the variable $\nomj$ (resp.~$\cnomm$) ranges over a completely join-generating (resp.~completely meet-generating) subset  of $\mathbb{A}^\delta$, and $\wdia$ is completely join-preserving,  we can equivalently rewrite the initial inequality as follows: $\forall p(\bigvee\{\wdia\wdia \nomj\mid\nomj\leq p\}\leq \bigwedge \{\cnomm\mid \wdia p\leq \cnomm\})$, which yields the required equivalence by the definition of the least upper bound and the greatest lower bound of subsets of a poset. The second equivalence is an instance of the core rule of ALBA, which allows to eliminate the quantification over  proposition variables. As to the direction from bottom to top, by the monotonicity of $\wdia$, the inequalities $\nomj\leq p$ and $\wdia p\leq \cnomm$ immediately imply $\wdia\nomj\leq \wdia p\leq \cnomm$, from which the  required inequality $\wdia\wdia \nomj\leq \cnomm$ follows by assumption.  For the converse direction, for a given interpretations of $\nomj$ and $\cnomm$ such that $\wdia \nomj\leq \cnomm$, we let $p$ have the same interpretation as $\nomj$. Then this interpretation satisfies both  inequalities $\nomj\leq p$ and $\wdia p = \wdia \nomj\leq \cnomm$, from which the  required inequality $\wdia\wdia \nomj\leq \cnomm$ follows by assumption. The third equivalence immediately follows from considerations similar to those made for justifying the first equivalence; namely, that the inequality $\wdia\wdia \nomj\leq \cnomm$ can be equivalently rewritten as $\bigvee\{\wdia \nomh\mid \nomh\leq \wdia \nomj\}\leq \cnomm$, which yields the required equivalence by the definition of a subset of a poset. Analogous arguments can be made to justify  the following chains of equivalences:
{\small
\begin{center}
    \begin{tabular}{rcll}
    && $\forall p(\wbox  p \leq   p)$\\
&  iff & $\forall p\forall \nomj  \forall \cnomm \, ((\nomj\leq \wbox p\ \&\  p\leq \cnomm)\Rarr  \nomj\leq \cnomm)$  & \ \ \ join- and meet-generation \\   
&  iff & $\forall \nomj  \forall \cnomm \, (\nomj\leq \wbox  \cnomm\Rarr  \nomj\leq \cnomm)  $ &  \ \ \ Ackermann's lemma\\
    \end{tabular}
\end{center}

\begin{center}
    \begin{tabular}{rcll}
    && $\forall p(  p \leq \wbox \wdia p)$\\
&  iff & $\forall p\forall \nomj  \forall \cnomm \, ((\nomj\leq  p\ \&\  \wdia p\leq \cnomm)\Rarr  \nomj\leq \wbox \cnomm)$  &  \ \ \ join- and meet-generation \\   
&  iff & $\forall \nomj  \forall \cnomm \, (\wdia \nomj\leq  \cnomm\Rarr  \nomj\leq \wbox \cnomm)  $ &  \ \ \ Ackermann's lemma\\
    \end{tabular}
\end{center}

\begin{center}
    \begin{tabular}{rcll}
    && $\forall p(\wbox  p \leq  \wdia p)$\\
&  iff & $\forall p\forall \nomj  \forall \cnomm \, ((\nomj\leq \wbox p\ \&\  \wdia p\leq \cnomm)\Rarr  \nomj\leq \cnomm)$  &  \ \ \ join- and meet-generation \\ 
&  iff & $\forall p\forall \nomj  \forall \cnomm \, ((\bdia\nomj\leq  p\ \&\  \wdia p\leq \cnomm)\Rarr  \nomj\leq \cnomm)$  &  \ \ \ $\bdia\dashv \wbox$ adjunction \\ 
&  iff & $\forall \nomj  \forall \cnomm \, (\wdia \bdia\nomj\leq \cnomm\Rarr  \nomj\leq \cnomm)  $ &  \ \ \ Ackermann's lemma\\
&  iff & $\forall \nomj  \forall \cnomm \exists \nomk\, ((\nomk\leq  \bdia\nomj\Rarr \wdia \nomk\leq \cnomm)\Rarr \nomj\leq \cnomm)  $ &  \ \ \ join-generation\\
    \end{tabular}
\end{center}

\begin{center}
    \begin{tabular}{rcll}
    && $\forall p(\wdia\wbox  p \leq  \wbox\wdia p)$\\
&  iff & $\forall p\forall \nomj  \forall \cnomm \, ((\nomj\leq \wbox p\ \&\  \wdia p\leq \cnomm)\Rarr  \wdia\nomj\leq \wbox\cnomm)$  &  \ \ \ join- and meet-generation \\ 
&  iff & $\forall p\forall \nomj  \forall \cnomm \, ((\bdia\nomj\leq  p\ \&\  \wdia p\leq \cnomm)\Rarr  \bdia\wdia\nomj\leq \cnomm)$  &  \ \ \ $\bdia\dashv \wbox$ adjunction \\ 
&  iff & $\forall \nomj  \forall \cnomm \, (\wdia \bdia\nomj\leq \cnomm\Rarr  \bdia\wdia\nomj\leq \cnomm)  $ &  \ \ \ Ackermann's lemma\\
&  iff & $\forall \nomj  \forall \cnomm \forall \nomh\exists \nomk\, ((\nomk\leq  \bdia\nomj\Rarr \wdia \nomk\leq \cnomm)\Rarr (\nomh\leq \wdia\nomj\Rarr \bdia\nomh\leq \cnomm))  $ &  \ \ \ join- and meet-generation\\
    \end{tabular}
\end{center}
 }
 
The second equivalence in the chain above is based on the existence of the adjoints of the maps interpreting the original connectives on canonical extensions of $\mathcal{L}$-algebras (cf.~Section \ref{ssec:interpretation ALBA}). 
Finally, we remark that
carrying out the correspondence arguments above in the algebraic environment of the canonical extensions of $\mathcal{L}$-algebras allows us to clearly identify their pivotal properties, and, in particular, to verify that no property related with the setting of (perfect) distributive lattices (viz.~the complete join-primeness of the elements interpreting nominal variables) is required.

\section{The labelled calculus A.L and some of its extensions}
\label{sec:LabelledCalculusAL}


In what follows, we use  $p, q, \ldots$ for proposition variables, $A, B, \ldots$ for {\em formulas} metavariables (in the original language of the logic), $\nomj, \nomi, \ldots$ for nominal variables, $\cnomm, \cnomn, \ldots$ for conominal variables, $\NOMJ, \NOMH, \ldots$ (resp.~$\CNOMM, \CNOMN, \ldots$) for {\em nominal terms} metavariables (resp. {\em conominal terms}), $\NCT$ for {\em terms} metavariables, and $\Gamma, \Delta, \ldots$ for {\em meta-structures} metavariables. Given $p \in \mathsf{Prop}$, $\nomj \in \mathsf{NOM}$, $\cnomm \in \mathsf{CNOM}$, the language of (labelled) formulas, terms and structures is defined as follows:
\begin{center}
\begin{tabular}{rcccl}
$\textrm{formulas}$ & \ $\ni$\  & $A$ & $\ ::=\ $ & $p \mid \aatop \mid \abot \mid A \aand A \mid A \aor A \mid \wbox A \mid \wdia A$ \\
$\textrm{nominal terms}$   & \ $\ni$\  & $\NOMJ$  & $\ ::=\ $ & $\nomj \mid \wdia \nomj  \mid \bdia \nomj$ \\
$\textrm{conominal terms}$ & \ $\ni$\  & $\CNOMM$ & $\ ::=\ $ & $\cnomm \mid \wbox \cnomm \mid \bbox \cnomm$ \\
$\textrm{terms}$ & \ $\ni$\  & $\NCT$ & $\ ::=\ $ & $\NOMJ \mid \CNOMM$ \\
$\textrm{labelled formulas}$ & \ $\ni$\  & $a$ & $\ ::=\ $ & $\nomj \leq A \mid A \leq \cnomm$ \\
$\textrm{pure structures}$ & \ $\ni$\  & $t^{\ast}$ & $\ ::=\ $ & $\nomj \leq \NCT \mid \NCT \leq \cnomm$ \\
$\textrm{structures}$ & \ $\ni$\  & $\sigma$ & $\ ::=\ $ & $a \mid t$ \\ 
$\textrm{meta-structures}$ & \ $\ni$\  & $\Gamma$ & $\ ::=\ $ & $\sigma \mid \Gamma, \Gamma$ \\
\mc{5}{c}{\rule[0mm]{0mm}{6mm} \fns ${}^\ast$Side condition: $\nomj$ and $\cnomm$ do not occur in $\NCT$.} \\ 
\end{tabular}
\end{center}

Let us first recall some terminology (see e.g.~\cite[Section 4.1]{Wa98}) and notation. A ($\mathbf{A.L}$-){\em sequent} is a pair $\Gamma \fCenter \Delta$ where $\Gamma$ and $ \Delta$ (the {\em antecedent} and the {\em consequent} of the sequent, respectively) are metavariables for meta-structures separated by commas. An {\em inference} $r$, also called an instance of a rule, is a pair $(S, s)$ of a (possibly empty) set of sequents $S$ (the premises) and a sequent $s$ (the conclusion). We identify a {\em rule} $R$ with the set of all instances that are instantiations of $R$. A {\em rule} $R$, also referred to as a {\em scheme} is usually presented schematically using metavariables for meta-structures (denoted by upper-case Greek letters $\Gamma, \Delta, \Pi, \Sigma, \ldots, \Gamma_1, \Gamma_2, \ldots$), or metavariables for structures (denoted by lower-case Latin letters: $a, b, c, \ldots, a_1, a_2, \ldots$ for labelled formulas and $t_1, t_2, t_3, \ldots$ for pure structures), or metavariables for formulas (denoted by $A, B, C, \ldots A_1, A_2, \ldots$,  or metavariables for terms (denoted by $\nomj, \nomi, \nomh, \ldots, \nomj_1, \nomj_2, \ldots$ for nominal terms and $\cnomm, \cnomn, \cnomo, \ldots, \cnomm_1, \cnomm_2, \ldots$ for conominal terms). A rule $R$ with no premises, i.e.~$S = \emptyset$, is called an {\em axiom scheme}, and an instantiation of such $R$ is called an {\em axiom}. The immediate subformulas of a principal formula (see Definition \ref{def:ParametersCongruenceHistoryTree}) in the premise(s) of an operational inference are called {\em auxiliary formulas}. The formulas that are not preserved in an inference instantiating the cut rule are called {\em cut formulas}. If the cut formulas are principal in an inference instantiating the cut rule, then the inference is called {\em principal cut}. A cut that is not principal is called {parametric}. A {\em proof} of (the instantiation of) a sequent $\Gamma \fCenter \Delta$ is a tree where (the instantiation of) $\Gamma \fCenter \Delta$ occurs as the end-sequent, all the leaves are (instantiations of) axioms, and each node is introduced via an inference. Before providing the list of the primitive rules of $\mathbf{A.L}$, we need two preliminary definitions. 

\begin{definition}[Analysis]
\label{def:ParametersCongruenceHistoryTree}
The {\em specifications} are instantiations of meta-structure meta-variables in the statement of $R$. The {\em parameters} of $r \in R$ are {substructures} of instantiations of \mbox{(meta-)}structure metavariables in the statement of $R$. A formula instance is {\em principal} in an inference $r \in R$ if it is not a parameter in the conclusion of $r$ (except for switch rules). 

(Meta-)Structure occurrences in an inference $r \in R$ are in the (symmetric) relation of {\em local congruence} in $r$ if they instantiate the same metavariable occurring in the same position in a premise and in the conclusion of $R$, or they instantiate nonparametric structures in the application of switch rules (namely, in the case of $\mathbf{A.L}\Sigma$, occurrences of labelled formulas $\nomj\leq A$ and $A\leq \cnomm$, or occurrences of pure structures $\nomj\leq \NCT$ and $\NCT\leq \cnomm$). Therefore, the local congruence is a relation between specifications. 

Two occurrences instatiating a (meta-)structure are in the {\em inference congruence relation} if they are locally congruent in an inference $r$ occurring in a proof $\pi$. The {\em proof congruent relation} is the transitive closure of the {\em inference congruence relation} in a derivation $\pi$. 
\end{definition}

\begin{definition}[Position]
\label{def:Position}
For any well-formed sequent $\Gamma \vd \Delta$, 
\begin{itemize}
\item The occurrence of a labelled formula $\nomj \leq A$ (resp.~$A \leq \cnomm$) is in {\em precedent} position if  $\nomj \leq A  \in \Gamma$ (resp.~$A \leq \cnomm \in \Delta$), and it is in {\em succedent} position if $\nomj \leq A  \in \Delta$ (resp.~$A \leq \cnomm \in \Gamma$);
\item  any occurrence of a pure structure $\nomj \leq \NCT$ in $\Gamma$ (resp.~$\Delta$) is in {\em precedent} (resp.~{\em succedent}) position; any occurrence of a pure structure $ \NCT\leq \cnomm$ in $\Gamma$ (resp.~$\Delta$) is in {\em succedent} (resp.~{\em precedent}) position.  
\end{itemize}
\end{definition}

We follow the notational conventions as stated in Definition \ref{def:ParametersCongruenceHistoryTree}, which provides the so-called {\em analysis} of the rules of any proper labelled calculus. In particular, according to Definition \ref{def:ParametersCongruenceHistoryTree}, notice that if an occurrence $\sigma$ is a substructure of $\Pi \in \{\Gamma, \Gamma', \Delta, \Delta'\}$ occurring in an instantiation $r$ of a rule $R \in \mathbf{A.L}$ (including axioms, namely rules with no premises), then $\sigma$ is a {\em parameter} of $r$ and every other $\sigma'$ is {\em nonparametric} in $r$;\footnote{Therefore, every instantiation of a labelled formula (resp.~a pure structure) occurring in $R \in \mathbf{A.L}$ is nonparametric. For instance, $\nomj \leq p$ is nonparametric in Id$_{\nomj \leq p}$, and $A \leq \cnomm, \nomj \leq \wbox A, \nomj \leq \wbox \cnomm$ are all nonparametric in $\wbox_P$. Moreover, according to Definition \ref{def:ParametersCongruenceHistoryTree}, every instantiation of a structure (resp.~a labelled formula) in the conclusion of initial rules (resp.~logical rules) is {\em principal}. For instance, $\nomj \leq p$ is principal in Id$_{\nomj \leq p}$ and $\nomj \leq \wbox A$ is principal in $\wbox_P$.}  moreover, if $\sigma$ occurs in a premise and in the conclusion of $r$ in the same position (namely, in {\em precedent} versus in {\em succedent position}: see Definition \ref{def:Position}), then these two occurrences of $\sigma$ are {\em locally congruent} in $r$.\footnote{For instance, given an instantiation $r$ of the rule $\aand_S$, assuming $\sigma \in \Gamma$ in the first (resp.~second) premise of $r$, then it occurs in the same position in the conclusion of $r$, and these two occurrences of $\sigma$ are {\em locally congruent} in $r$. Nonetheless, notice that the two occurrences of $\sigma$ in the premises of $r$ are not locally congruent in $r$.} Notice that in the display calculi literature `being locally congruent' usually presupposes `being parametric', but in labelled calculi this is not anymore the case due to the presence of switch rules (see Remark \ref{AnalysisOfSwitchRules}).

{\small
\vspace{-0.7 cm}
\begin{center}
\begin{tabular}{rl}
\mc{2}{c}{\rule[-1.85mm]{0mm}{8mm}\textbf{Initial rules$^\ast$}} \\
\AXC{$ \ $}
\LL{Id$_{\,\nomj \leq p}$}
\UIC{$\nomj \leq p \vd \nomj \leq p$}
\DP
 \ \,
\AXC{$\ $}
\RL{Id$_{\,p \leq \cnomm}$}
\UIC{$p \leq \cnomm \vd p \leq \cnomm$}
\DP 

\  & \ 
 
\AXC{ \ }
\LL{\fns Id$_{\bot}$}
\UIC{$\bot \leq \cnomm \vd \bot \leq \cnomm$}
\DP
 \ \,
\AXC{ \ }
\RL{\fns Id$_{\top}$}
\UIC{$\nomj \leq \top \vd \nomj \leq \top$}
\DP
 \\

%
\end{tabular}
\end{center}
}

The initial rules above encode identities for atomic propositions and zeroary connectives, namely the fact that the derivability relation $\vd$ is reflexive. Identity sequents of the form $\nomj \leq A \vd \nomj \leq A$ (resp.~$A \leq \cnomm \vd A \leq \cnomm$) are derivable in the calculus.



{\small
\vspace{-0.2 cm}
\begin{center}
\begin{tabular}{rl}
\mc{2}{c}{\rule[-1.85mm]{0mm}{8mm}\textbf{Initial rules for $\aatop$ and $\abot$$^\ast$}} \\
\AXC{ \ }
\LL{\fns $\abot_{\nomj}$}
\UIC{$\nomj \leq \abot \vd \nomj \leq A$}
\DP
 \ \,
\AXC{ \ }
\LL{\fns $\abot_{\cnomm}$}
\UIC{$B \leq \cnomm \vd \abot \leq \cnomm$}
\DP
\ & \
\AXC{ \ }
\RL{\fns $\aatop_{\cnomm}$}
\UIC{$\aatop \leq \cnomm \vd B \leq \cnomm$}
\DP
 \ \,
\AXC{ \ }
\RL{\fns $\aatop_{\nomj}$}
\UIC{$\nomj \leq A \vd \nomj \leq \aatop$}
\DP
 \\
\mc{2}{c}{\rule[0mm]{0mm}{6mm} \fns ${}^\ast$Side condition: $A \in \{p, A_1 \aand A_2, \wbox A_1\}$ and $B \in \{p, B_1 \aor B_2, \wdia B_1\}$} \\
\end{tabular}
\end{center}
}

The initial rules for $\abot$ (resp.~$\aatop$) above encodes the fact that $\abot$ is interpreted as the minimal element (resp.~$\aatop$ as the maximal element) in the algebraic interpretation. 

The cut rules below encode the fact that the derivability relation $\vd$ is transitive. Notice that the notion of `cut formula' in standard Gentzen sequent calculi corresponds to `labelled cut formula' in the present setting. Before defining the cut rules, we need the following definition.

\begin{definition}
\label{def:j-m-LabelledFormulas}
A labelled formula $a$ is a $\nomj$-labelled (resp.~$\cnomm$-labelled) formula in a derivation $\pi$ if the uppermost labelled formulas congruent with $a$ in $\pi$ are introduced via $\textrm{Id}_{\nomj\leq p}, \abot_{\nomj}, \aatop_{\nomj}, \aand_P, \aand_S, \wbox_P, \wbox_S$ (resp.~$\textrm{Id}_{p \leq \cnomm}, \abot_{\cnomm}, \aatop_{\cnomm}, \aor_P, \aor_S, \wdia_P, \wdia_S$). 
\end{definition}

{\small
\vspace{-0.3 cm}
\begin{center}
\begin{tabular}{rl}
\mc{2}{c}{\rule[-1.85mm]{0mm}{8mm}\textbf{Cut rules}$^\ast$} \\
\AXC{$\Gamma \vd \nomj \leq A, \Delta$}
\AXC{$\Gamma', \nomj \leq A \vd \Delta'$}
\LL{Cut$_{\,\nomj\leq A}$}
\BIC{$\Gamma, \Gamma' \vd \Delta, \Delta'$}
\DP
 \ & \  
\AXC{$\Gamma \vd B \leq \cnomm, \Delta$}
\AXC{$\Gamma', B \leq \cnomm \vd \Delta'$}
\RL{Cut$_{\,B\leq \cnomm}$}
\BIC{$\Gamma, \Gamma' \vd \Delta, \Delta'$}
\DP 
 \\
\mc{2}{c}{\rule[0mm]{0mm}{6mm} \fns ${}^\ast$Side condition: $\nomj \leq A$ and $B \leq \cnomm$ are in display,} \\
\mc{2}{c}{$\nomj \leq A$ is a $\nomj$-labelled formula and $B \leq \cnomm$ is an $\cnomm$-labelled formula.} \\ 
\end{tabular}
\end{center}
 }


The switch rules below encode elementary properties of pairs of inequalities with the same approximant (either a nominal $\nomj$ or a conominal $\cnomm$) occurring in the same sequent with opposite polarity (namely, the first in precedent position and the second in succedent position: see Definition \ref{def:Position}). Notice that we might use $S$ as a generic name denoting a specific switch rule in the following set. If so, we rely on the context to disambiguate which rule we are referring to. In particular, the label $S$ is unambiguous whenever we use it as the name for a rule application in a derivation.

{\small
\begin{center}
\begin{tabular}{rl}
\mc{2}{c}{\textbf{Switch rules$^\ast$}}  \\
\AXC{$\Gamma,  \nomj \leq A  \vd  \nomj  \leq \cnomm,  \Delta$}
\LL{\fns S$\cnomm$}
\UIC{$\Gamma \vd A  \leq \cnomm ,  \Delta$}
\DP 
 \ & \ 
\AXC{$\Gamma, A \leq \cnomm \vd  \nomj  \leq \cnomm,  \Delta$}
\RL{\fns S$\nomj$}
\UIC{$\Gamma \vd \nomj \leq A,  \Delta$}
\DP 
 \\

 & \\

\AXC{$\Gamma, \nomj \leq  A \vd \nomj \leq B, \Delta$}
\LL{S$\cnomm\cnomm$}
\UIC{$\Gamma, B \leq \cnomm \vd A \leq  \cnomm, \Delta$}
\DP
 \ & \ 
\AXC{$\Gamma, A \leq \cnomm \vd B \leq  \cnomm, \Delta$}
\RL{S$\nomj\nomj$}
\UIC{$\Gamma, \nomj \leq  B \vd \nomj \leq A, \Delta$}
\DP
 \\

 & \\

\AXC{$\Gamma,\nomj \leq \NCT \vd \nomj \leq A, \Delta$}
\LL{S$\cnomm\NCT$}
\UIC{$\Gamma,A \leq \cnomm \vd \NCT \leq \cnomm, \Delta$}
\DP

 \ & \  

\AXC{$\Gamma, \NCT \leq \cnomm \vd  A \leq \cnomm, \Delta$}
\RL{S$\nomj\NCT$}
\UIC{$\Gamma, \nomj \leq A \vd \nomj \leq \NCT, \Delta$}
\DP

 \\

 & \\
  \AXC{$\Gamma,\nomj \leq A \vd \nomj \leq \NCT, \Delta$}
\LL{S$\NCT\cnomm$}
\UIC{$\Gamma, \NCT \leq \cnomm \vd A \leq \cnomm, \Delta$}
\DP

 \ & \  

\AXC{$\Gamma, A \leq \cnomm \vd  \NCT \leq \cnomm, \Delta$}
\RL{S$\NCT\nomj$}
\UIC{$\Gamma, \nomj \leq \NCT \vd \nomj \leq A, \Delta$}
\DP
 
 \\
 &\\
  \AXC{$\Gamma,\nomj \leq \NCT' \vd \nomj \leq \NCT, \Delta$}
\LL{S$\NCT\NCT'\cnomm$}
\UIC{$\Gamma, \NCT \leq \cnomm \vd \NCT' \leq \cnomm, \Delta$}
\DP

 \ & \  

\AXC{$\Gamma, \NCT' \leq \cnomm \vd  \NCT \leq \cnomm, \Delta$}
\RL{S$\nomj\NCT\NCT'$}
\UIC{$\Gamma, \nomj \leq \NCT \vd \nomj \leq \NCT', \Delta$}
\DP
 \\

 





\mc{2}{c}{\rule[0mm]{0mm}{6mm} ${ }^\ast$Side condition: \fns{For all the switch rules except S$\cnomm$ and S$\nomj$,}} \\
\mc{2}{c}{\fns {$\nomj$ and $\cnomm$ do not appear in $\Gamma$ or $\Delta$. $\nomj$ (resp.~$\cnomm$) in S$\cnomm$ (resp.~S$\nomj$)}}\\
\mc{2}{c}{\fns {must not appear in the conclusion of the rule.}}
\end{tabular}
\end{center}
 }

\begin{remark}[Analysis of switch rules]
\label{AnalysisOfSwitchRules}
For each instantiation $r$ of $R \in \{\textrm{S}\NCT\NCT'\cnomm, \textrm{S}\nomj\NCT\NCT'\}$, the instantiations of $\nomj \leq \NCT''$ and $\NCT'' \leq \cnomm$ (where $\NCT'' \in \{\NCT, \NCT'\}$) are {\em nonparametric} and {\em locally congruent} in $r$ (see Definition \ref{def:ParametersCongruenceHistoryTree}). For each instantiation $r$ of any other switch rule $R$, the instantiations of $\nomj \leq C$ and $C \leq \cnomm$ (where $C \in \{A, B\}$) are {\em nonparametric} and {\em locally congruent} in $r$ (see Definition \ref{def:ParametersCongruenceHistoryTree}).
\end{remark}

%
%
%

{\small
\begin{center}
\begin{tabular}{rl}

\mc{2}{c}{\rule[-1.85mm]{0mm}{8mm} \textbf{Adjunction rules}\rule[-1.85mm]{0mm}{8mm}} \\

\AXC{$\Gamma \vd \wdia \nomj  \leq \cnomm,  \Delta$}
\LL{\fns $\wdia \dashv \bbox$}
\UIC{$\Gamma \vd \nomj \leq \bbox \cnomm,  \Delta$}
\DP 
 \ & \  
\AXC{$\Gamma \vd \nomj \leq \bbox \cnomm,  \Delta$}
\RL{\fns $\wdia \dashv \bbox^{-1}$}
\UIC{$\Gamma \vd \wdia \nomj \leq\cnomm,  \Delta$}
\DP
 \\

 & \\
 
\AXC{$\Gamma \vd \nomj \leq \wbox \cnomm,  \Delta$}
\LL{\fns $\bdia \dashv \wbox$}
\UIC{$\Gamma \vd \bdia \nomj \leq \cnomm,  \Delta$}
\DP 
 \ & \  
\AXC{$\Gamma \vd \bdia \nomj \leq \cnomm,  \Delta$}
\RL{\fns $\bdia \dashv \wbox^{-1}$}
\UIC{$\Gamma \vd  \nomj \leq \wbox \cnomm,  \Delta$}
\DP
 \\
\end{tabular}
\end{center}
 }

{\small
\begin{center}
\begin{tabular}{rl}
\mc{2}{c}{\rule[-1.85mm]{0mm}{8mm} \textbf{Structural rules for $\aatop$ and $\abot$}\rule[-1.85mm]{0mm}{8mm}} \\
%

\AXC{$\Gamma \vd \aatop \leq \cnomm, \Delta$}
\LL{\fns $\aatop_\wbox$}
\UIC{$\Gamma, \nomj \leq \aatop \vd \nomj \leq \wbox \cnomm, \Delta$}
\DP
 \ & \ 
\AXC{$\Gamma \vd \nomj \leq \abot, \Delta$}
\RL{\fns $\abot_\wdia$}
\UIC{$\Gamma, \abot \leq \cnomm \vd \wdia \nomj \leq \cnomm, \Delta$}
\DP
 \\
\end{tabular}
\end{center}
 }

The adjunction rules above encode the fact that unary modalities $\wdia, \bbox$ and $\bdia, \wbox$ constitute pairs of adjoint  operators. The structural rules $\aatop_\wbox$ (resp.~$\abot_\wdia$) above encodes the fact that $\wbox$ preserve $\aatop$ (resp.~$\wdia$ preserves $\abot$). 

The logical rules below encode the minimal order-theoretic properties and the arity of propositional
and modal connectives.


{\small
\begin{center}
\begin{tabular}{rl}

\mc{2}{c}{\rule[-1.85mm]{0mm}{8mm} \textbf{Logical rules for propositional connectives$^{\ast}$}\rule[-1.85mm]{0mm}{8mm}} \\

\AXC{$\Gamma, \nomj \leq  A_i \vd \Delta$}
\LL{$\aand_P$}
\UIC{$\Gamma, \nomj \leq A_1 \aand A_2 \vd \Delta$}
\DP
 \ & \ 
\AXC{$\Gamma \vd \nomj \leq A, \Delta$}
\AXC{$\Gamma \vd \nomj \leq B, \Delta$}
\RL{$\aand_S$}
\BIC{$\Gamma \vd \nomj \leq A \aand B, \Delta$}
\DP
 \\

 & \\
 
\AXC{$\Gamma \vd A \leq \cnomm, \Delta$}
\AXC{$\Gamma \vd B \leq \cnomm, \Delta$}
\LL{$\aor_P$}
\BIC{$\Gamma \vd A \aor B \leq \cnomm, \Delta$}
\DP
 \ & \  
\AXC{$\Gamma, A_i \leq \cnomm \vd \Delta$}
\RL{$\aor_S$}
\UIC{$\Gamma, A_1 \aor A_2 \leq \cnomm \vd \Delta$}
\DP
 \\
\mc{2}{c}{\rule[0mm]{0mm}{6mm} \fns ${}^\ast$Side condition: labelled formula in the conclusion of any logical rule are in display.} \\
\end{tabular}
\end{center}
 }

We consider the following logical rules for modalities, where $\wbox_S$ and $\wdia_P$ are invertible, but $\wbox_P$ and $\wdia_S$ are not.  This choice facilitates a smoother analysis of the rules and therefore it is preferable whenever the goal is to provide a canonical cut elimination. 

{\small
\begin{center}
\begin{tabular}{rl}

\mc{2}{c}{\rule[-1.85mm]{0mm}{8mm} \textbf{Logical rules for modalities$^{\ast}$}\rule[-1.85mm]{0mm}{8mm}} \\

\AXC{$\Gamma \vd A \leq \cnomm, \Delta$}
\LL{$\wbox_P$}
\UIC{$\Gamma, \nomj \leq \wbox A \vd \nomj \leq \wbox \cnomm, \Delta$}
\DP 
 \ & \  
\AXC{$\Gamma, A \leq \cnomm \vd \nomj \leq \wbox \cnomm, \Delta$}
\RL{$\wbox_S$}
\UIC{$\Gamma \vd \nomj \leq \wbox A, \Delta$}
\DP
 \\

 & \\

\AXC{$\Gamma, \nomj \leq A \vd \wdia \nomj \leq \cnomm, \Delta$}
\LL{$\wdia_P$}
\UIC{$\Gamma\vd \wdia A \leq \cnomm, \Delta$}
\DP
 \ & \ 
\AXC{$\Gamma \vd \nomj \leq A, \Delta$}
\RL{$\wdia_S$}
\UIC{$\Gamma, \wdia A \leq \cnomm \vd \wdia \nomj \leq \cnomm, \Delta$}
\DP  
 \\
\end{tabular}
\begin{tabular}{rl}
\mc{2}{c}{\rule[0mm]{0mm}{6mm} \fns $^\ast$Side conditions: $\cnomm$ (resp.~$\nomj$) must not occur in the conclusion of $\wbox_S$ (resp.~$\wdia_P$).} \\
\mc{2}{c}{\fns Labelled formulas in the conclusion of any logical rule are in display.} \\
\end{tabular}
\end{center}
 }


\begin{remark}
The invertible version of $\aand_P$,  $\aor_S$, f $\wbox_P$ and $\wdia_S$  are  as follows:

{\small
\begin{center}
\begin{tabular}{rl}
\AXC{$\Gamma, \nomj \leq  A, \nomj \leq B \vd \Delta$}
\LL{$\aand_P$}
\UIC{$\Gamma, \nomj \leq A \aand B \vd \Delta$}
\DP
 \ & \ 
\AXC{$\Gamma, A \leq \cnomm, B \leq \cnomm \vd \Delta$}
\RL{$\aor_S$}
\UIC{$\Gamma, A \aor B \leq \cnomm \vd \Delta$}
\DP
\\
 & \\
\AXC{$\Gamma, \nomj \leq \wbox A \vd A \leq \cnomm, \nomj \leq \wbox \cnomm, \Delta$}
\LL{$\wbox_P$}
\UIC{$\Gamma, \nomj \leq \wbox A \vd \nomj \leq \wbox \cnomm, \Delta$}
\DP 
 \ & \  
\AXC{$\Gamma, \wdia A\leq \cnomm \vd \nomj \leq A, \wdia \nomj \leq \cnomm, \Delta$}
\RL{$\wdia_S$}
\UIC{$\Gamma, \wdia A \leq \cnomm \vd \wdia \nomj \leq \cnomm, \Delta$}
\DP 
 \\
\end{tabular}
\end{center}
 }
 
Invertible rules can be used whenever the goal is to facilitate backwards-looking proof searches. In this case, the initial rules have to be generalized accordingly.
\end{remark}


The table below collects the analytic rules, both in the format of display calculi and in the present format,  generated by reading off the ALBA outputs of the corresponding axioms  reported on in Section \ref{sec:TheAlgorithmALBA}.

{\footnotesize 
\begin{center}
\begin{tabular}{|rc|c|c|}
\hline
\mc{2}{c}{modal axiom} & \mc{1}{c}{display rule} & \mc{1}{c}{labelled rule} \\
\hline
\rule[-4mm]{0mm}{10mm}
(4) & $\wdia \wdia A \vdash \wdia A$ & 
\AXC{$\WDIA X \vd Y$}
\LL{\fns $4$}
\UIC{$\WDIA \WDIA X \vd Y$}
\DP & 
\AXC{$\Gamma \vd \wdia \nomj \leq \cnomm, \Delta$}
\LL{\fns $4$}
\UIC{$\Gamma, \nomh \leq \wdia \nomj \vd \wdia \nomh \leq \cnomm, \Delta$}
\DP \\
\hline
\rule[-4mm]{0mm}{10mm}
(T) & $\wbox A \vdash A$ & 
\AXC{$X \vd \WBOX Y$}
\RL{\fns $T$}
\UIC{$X \vd Y$}
\DP
&  
\AXC{$\Gamma \vd \nomj \leq \wbox \cnomm, \Delta$}
\RL{\fns $T$}
\UIC{$\Gamma \vd \nomj \leq \cnomm, \Delta$}
\DP
\\
\hline
\rule[-4mm]{0mm}{10mm}
(B) & $A \vdash \wbox \wdia A$ &
\AXC{$\WDIA X \vd Y$}
\RL{\fns $B$}
\UIC{$X \vd \WBOX Y$}
\DP
 & 
\AXC{$\Gamma \vd \wdia \nomj\leq \cnomm, \Delta$}
\RL{\fns $B$}
\UIC{$\Gamma \vd \nomj\leq \wbox \cnomm, \Delta$}
\DP
 \\
\hline
\rule[-4mm]{0mm}{10mm}
(D) & $ \wbox A \vdash \wdia A$ &
\AXC{$\WDIA \BDIA X \vd Y$}
\LL{\fns D}
\UIC{$ X \vd Y$}
\DP
& 
\AXC{$\Gamma, \nomk \leq \bdia \nomj \vd \wdia \nomk \leq \cnomm, \Delta$}
\LL{\fns D}
\UIC{$\Gamma \vd \nomj\leq \cnomm, \Delta$}
\DP
 \\
\hline
\rule[-4mm]{0mm}{10mm}
(C) & $\wdia \wbox A \vdash \wbox \wdia A$ \ \ & \ \  
\AXC{$\WDIA \BDIA X \vd Y$}
\LL{\fns $C$}
\UIC{$\BDIA \WDIA X \vd Y$}
\DP \ \ \ 
& 
\ \ 
\AXC{$\Gamma, \nomk\leq \bdia\nomj \vd \wdia \nomk\leq \cnomm, \Delta$}
\LL{\fns $C$}
\UIC{$\Gamma, \nomh\leq \wdia\nomj \vd \bdia\nomh\leq \cnomm, \Delta$}
\DP \ \ 
\\
\hline
\end{tabular}
\end{center}
 }
 
 Where $\nomk$ in rules $C$ and $D$ must not appear in $\Gamma, \Delta$. For any $\Sigma\subseteq \{(T), (4), (B), (D), (C)\}$, we let $\mathbf{A.L}\Sigma$ be the calculus defined by the rules in  the sections above plus the additional rules in the table above corresponding to the axioms in $\Sigma$. We let $\mathbf{A.L}: = \mathbf{A.L}\varnothing$.

\section{Properties of the calculus $\mathbf{A.L}\Sigma$}
\label{sec:PropertiesOfTheValculusAL}
\subsection{Soundness}
In the present section, we show that, for any $\Sigma\subseteq \{(T), (4), (B), (D), (C)\}$, the rules of $\mathbf{A.L}\Sigma$ are sound on the class $\mathsf{K}^\delta(\Sigma): = \{\mathbb{A}^\delta\mid \mathbb{A}\models \Sigma\}$.
Firstly, let us recall that, as usual, any $\mathbf{A.L}$-sequent $\Gamma\vd \Delta$ is to be interpreted as ``any assignment of the variables in $\Prop\cup\mathsf{NOM}\cup\mathsf{CNOM}$ under which  all inequalities in $\Gamma$ are satisfied also satisfies some inequality in $\Delta$''; in symbols: $\forall \overline{p}\forall \overline{\nomj}\forall \overline{\cnomm}(\bigwith \Gamma   \Longrightarrow  \bigparr \Delta)$.

As to the basic calculus $\mathbf{A.L}$, the soundness of the initial rules, cut rules, adjunction rules, and logical rules for propositional connectives is straightforward. The soundness of the switch rules hinges on the fact that nominal and co-nominal variables range over completely join-generating and completely meet-generating subsets of $\mathbb{A}^\delta$ for any $\mathcal{L}$-algebra $\mathbb{A}$. For example,  the soundness of S$\nomj\cnomm$ follows from the following chain of equivalences:

{\small
\begin{center}
\begin{tabular}{cll}
&$\forall \ol{p}\forall \nomj \forall \ol{\nomk} \forall \ol{\cnomn} (\bigwith \Gamma \with \nomj \leq A \Rightarrow \nomj \leq B \parr \bigparr \Delta)$ &  \ \ \ validity of premise \\
iff & $\forall \ol{p} \forall \ol{\nomk} \forall \ol{\cnomn} (\bigwith \Gamma  \Rightarrow \forall \nomj (\nomj \leq A \Rightarrow \nomj \leq B) \parr \bigparr \Delta)$ &  \ \ \ uncurrying + side condition \\
iff & $\forall \ol{p} \forall \ol{\nomk} \forall \ol{\cnomn} (\bigwith \Gamma  \Rightarrow  A \leq B \parr \bigparr \Delta)$ &  \ \ \ c.~join generation \\
iff & $\forall \ol{p} \forall \ol{\nomk} \forall \ol{\cnomn} (\bigwith \Gamma  \Rightarrow \forall \cnomm (B \leq \cnomm \Rightarrow A \leq \cnomm) \parr \bigparr \Delta)$ &  \ \ \ c.~meet generation\\
iff & $\forall \ol{p} \forall \cnomm \forall \ol{\nomk} \forall \ol{\cnomn} (\bigwith \Gamma  \with B \leq \cnomm \Rightarrow  A \leq \cnomm \parr \bigparr \Delta)$ &  \ \ \ currying\\ 
\end{tabular}
\end{center}
}

 Since $\cnomm$ does not occur in  $\Gamma$ and $\Delta$, the rule is also invertible. The verification of the soundness of the remaining switch rules is similar.  

The soundness and invertibility of the introduction rules for the modal connectives  hinge on the fact that the operation $\wdia^\sigma$ (resp.~$\wbox^\pi$) is completely join-preserving (resp.~completely meet preserving)  and on the complete- join-generation and meet-generation properties of the subsets of $\mathbb{A}^\delta$ on which nominal and co-nominal variables are interpreted. For example,  the soundness and invertibility of $\wbox_S$ is verified via  the following chain of equivalences:

{\small
\begin{center}
\begin{tabular}{cll}
&$\forall \ol{p}\forall \nomj\forall \cnomm \forall \ol{\nomk} \forall \ol{\cnomn} (\bigwith \Gamma \with   A \leq \cnomm \Rightarrow  \nomj\leq \wbox \cnomm \parr \bigparr \Delta)$ &  \ \ \ validity of premise \\
iff & $\forall \ol{p} \forall \nomj \forall \ol{\nomk} \forall \ol{\cnomn} (\bigwith \Gamma  \Rightarrow \forall \cnomm ( A\leq \cnomm \Rightarrow \nomj \leq \wbox \cnomm) \parr \bigparr \Delta)$ &  \ \ \ uncurrying + side condition \\
iff & $\forall \ol{p} \forall \nomj \forall \ol{\nomk} \forall \ol{\cnomn} (\bigwith \Gamma  \Rightarrow \forall \cnomm ( A\leq \cnomm \Rightarrow \bdia \nomj \leq  \cnomm) \parr \bigparr \Delta)$ &  \ \ \ adjunction \\
iff & $\forall \ol{p} \forall \nomj \forall \ol{\nomk} \forall \ol{\cnomn} (\bigwith \Gamma  \Rightarrow  \bdia \nomj \leq A  \parr \bigparr \Delta)$ &  \ \ \ c.~meet-generation \\
iff & $\forall \ol{p} \forall \nomj \forall \ol{\nomk} \forall \ol{\cnomn} (\bigwith \Gamma  \Rightarrow   \nomj \leq \wbox A  \parr \bigparr \Delta)$ &  \ \ \ adjunction \\
\end{tabular}
\end{center}
}

The soundness of $\wbox_P$ immediately follows from the monotonicity of $\wbox^\pi$. Indeed, fix an assignment of variables in $\Prop\cup \mathsf{NOM}\cup\mathsf{CNOM}$ under which all inequalities in $\Gamma$ and $\nomj\leq \wbox A$ are satisfied. If such assignment also satisfies $A\leq \cnomn$, then, by monotonicity, $\nomj\leq \wbox A\leq \wbox \cnomn$, as required. The proof for the rules $\wdia_P$ and $\wdia_S$ is similar.

As to the extended calculus $\mathbf{A.L}\Sigma$, the soundness of  rule (4) is verified by the following chain of computations holding on the canonical extension of any $\mathcal{L}$-algebra $\mathbb{A}$ such that $\mathbb{A}\models \wdia\wdia p\leq\wdia p$:

{\small
\begin{center}
\begin{tabular}{cll}
&$\forall \ol{p}\forall \ol{\nomk}\forall \nomj\forall \cnomm    \forall \ol{\cnomn} (\bigwith \Gamma  \Rightarrow  \wdia \nomj \leq \cnomm \parr \bigparr \Delta)$ & \ \ \ validity of premise \\
then &$\forall \ol{p}\forall \ol{\nomk}\forall \nomj\forall \cnomm \forall \ol{\cnomn} (\bigwith \Gamma  \Rightarrow  \wdia\wdia \nomj \leq \cnomm \parr \bigparr \Delta)$ & \ \ \ axiom (4) \\
iff & $\forall \ol{p}\forall \ol{\nomk}\forall \nomj  \forall \cnomm  \forall \ol{\cnomn} (\bigwith \Gamma \Rightarrow    \forall \nomh(\nomh \leq \wdia\nomj \Rightarrow  \wdia \nomh\leq  \cnomm )\parr \bigparr \Delta)$ & \ \ \ c.~join-generation \\
iff &$\forall \ol{p}\forall \ol{\nomk}\forall \nomj\forall \cnomm  \forall \nomh  \forall \ol{\cnomn} (\bigwith \Gamma \with \nomh \leq \wdia\nomj \Rightarrow  \wdia\nomh\leq  \cnomm \parr \bigparr \Delta)$ & \ \ \ currying + $\nomh$ fresh \\
\end{tabular}
\end{center}
 }
 
The key step in the computation above is the one which makes use of the assumption of  axiom (4) being valid on $\mathbb{A}$. Indeed, by the general theory of correspondence for LE-logics (cf.~\cite{CP-nondist}), axiom (4) is canonical, hence the assumption implies that  (4) is valid also on $\mathbb{A}^\delta$. Then, as is shown in the computation concerning axiom (4) in Section \ref{sec:TheAlgorithmALBA}, the validity of (4) in $\mathbb{A}^\delta$ is equivalent to the condition $\forall \nomj  \forall \cnomm \, (\wdia \nomj \leq \cnomm \Rarr \wdia \wdia \nomj  \leq \cnomm)$ holding in $\mathbb{A}^\delta$, which justifies this key step. The verification of the remaining additional rules  hinges on similar arguments and facts (in particular, all axioms we consider are canonical), so in what follows we only report on the corresponding computations.

{\small
\begin{center}
\begin{tabular}{cll}
&$\forall \ol{p}\forall \ol{\nomk}\forall \nomj\forall \cnomm    \forall \ol{\cnomn} (\bigwith \Gamma  \Rightarrow   \nomj \leq \wbox\cnomm \parr \bigparr \Delta)$ & \ \ \ validity of premise \\
then &$\forall \ol{p}\forall \ol{\nomk}\forall \nomj\forall \cnomm    \forall \ol{\cnomn} (\bigwith \Gamma  \Rightarrow   \nomj \leq \cnomm \parr \bigparr \Delta)$ & \ \ \ axiom (T) \\
\end{tabular}
\end{center}

\begin{center}
\begin{tabular}{cll}
&$\forall \ol{p}\forall \ol{\nomk}\forall \nomj\forall \cnomm    \forall \ol{\cnomn} (\bigwith \Gamma  \Rightarrow  \wdia \nomj \leq \cnomm \parr \bigparr \Delta)$ & \ \ \ validity of premise \\
then &$\forall \ol{p}\forall \ol{\nomk}\forall \nomj\forall \cnomm    \forall \ol{\cnomn} (\bigwith \Gamma  \Rightarrow   \nomj\leq \wbox \cnomm \parr \bigparr \Delta)$ & \ \ \ axiom (B) \\
\end{tabular}
\end{center}

\begin{center}
\begin{tabular}{cll}
&$\forall \ol{p}\forall \ol{\nomk}\forall \nomj\forall \cnomm    \forall \ol{\cnomn} (\bigwith \Gamma  \Rightarrow \nomj\leq \wbox \cnomm    \parr \bigparr \Delta)$ & \ \ \ validity of premise \\
then &$\forall \ol{p}\forall \ol{\nomk}\forall \nomj\forall \cnomm    \forall \ol{\cnomn} (\bigwith \Gamma  \Rightarrow  \wdia \nomj \leq \cnomm \parr \bigparr \Delta)$ & \ \ \ axiom (B$^{-1}$) \\
\end{tabular}
\end{center}
 }
 
 The soundness of the rule (C) is verified by the following chain of computations holding on the canonical extension of any $\mathcal{L}$-algebra $\mathbb{A}$ such that $\mathbb{A}\models \wdia\wbox p\leq\wbox\wdia p$:

{\small
\begin{center}
\begin{tabular}{cll}
&$\forall \ol{p}\forall \ol{\nomh'}\forall \nomj\forall \cnomm \forall \nomk   \forall \ol{\cnomn} (\bigwith \Gamma \with   \nomk \leq \bdia\nomj \Rightarrow  \wdia\nomk\leq  \cnomm \parr \bigparr \Delta)$ &  \ \ \ validity of premise \\
iff & $\forall \ol{p}\forall \ol{\nomh'}\forall \nomj\forall \cnomm    \forall \ol{\cnomn} (\bigwith \Gamma \Rightarrow    \forall \nomk(\nomk \leq \bdia\nomj \Rightarrow  \wdia\nomk\leq  \cnomm )\parr \bigparr \Delta)$ &  \ \ \ uncurrying + side cond. \\
iff &$\forall \ol{p}\forall \ol{\nomh'}\forall \nomj  \forall \cnomm  \forall \ol{\cnomn} (\bigwith \Gamma \Rightarrow    \forall \nomk(\nomk \leq \bdia\nomj \Rightarrow  \nomk\leq \bbox \cnomm )\parr \bigparr \Delta)$ &  \ \ \ adjunction\\

iff & $\forall \ol{p}\forall \ol{\nomh'}\forall \nomj  \forall \cnomm  \forall \ol{\cnomn} (\bigwith \Gamma \Rightarrow     \bdia\nomj \leq \bbox \cnomm \parr \bigparr \Delta)$ &  \ \ \ c.~join-generation \\
iff & $\forall \ol{p}\forall \ol{\nomh'}\forall \nomj\forall \cnomm    \forall \ol{\cnomn} (\bigwith \Gamma \Rightarrow     \wdia\bdia\nomj \leq  \cnomm \parr \bigparr \Delta)$ &  \ \ \ adjunction \\
then & $\forall \ol{p}\forall \ol{\nomh'}\forall \nomj  \forall \cnomm  \forall \ol{\cnomn} (\bigwith \Gamma \Rightarrow     \bdia \wdia\nomj \leq  \cnomm \parr \bigparr \Delta)$ &  \ \ \ axiom (C) \\
iff & $\forall \ol{p}\forall \ol{\nomh'}\forall \nomj\forall \cnomm   \forall \ol{\cnomn} (\bigwith \Gamma \Rightarrow     \wdia\nomj \leq \wbox \cnomm \parr \bigparr \Delta)$ &  \ \ \ adjunction \\
iff &$\forall \ol{p}\forall \ol{\nomh'}\forall \nomj\forall \cnomm    \forall \ol{\cnomn} (\bigwith \Gamma \Rightarrow    \forall \nomi(\nomi \leq \wdia\nomj \Rightarrow  \nomi\leq \wbox \cnomm )\parr \bigparr \Delta)$ &  \ \ \ c.~join-generation\\
iff &$\forall \ol{p}\forall\nomi\forall \ol{\nomh'}\forall \nomj \forall \cnomm  \forall \ol{\cnomn} (\bigwith \Gamma \with    \nomi \leq \wdia\nomj \Rightarrow  \nomi\leq \wbox \cnomm \parr \bigparr \Delta)$ &  \ \ \ currying\\
iff &$\forall \ol{p}\forall\nomi\forall \ol{\nomh'}\forall \nomj\forall \cnomm    \forall \ol{\cnomn} (\bigwith \Gamma \with    \nomi \leq \wdia\nomj \Rightarrow  \bdia\nomi\leq  \cnomm \parr \bigparr \Delta)$ &  \ \ \ adjunction\\
\end{tabular}
\end{center}
}

Instantiating $\nomi$ as $\nomh$ completes the proof. The key step in the computation above is the one which makes use of the assumption of  axiom (C) being valid on $\mathbb{A}$. Indeed, by the general theory of correspondence for LE-logics (cf.~\cite{CP-nondist}), axiom (C) is canonical, hence the assumption implies that  (C) is valid also on $\mathbb{A}^\delta$. Then, as is shown in the computation concerning axiom (C) in Section \ref{sec:TheAlgorithmALBA}, the validity of (C) in $\mathbb{A}^\delta$ is equivalent to the condition $\forall \nomj  \forall \cnomm \, (\wdia\bdia \nomj \leq \cnomm \Rarr \bdia \wdia \nomj  \leq \cnomm)$ holding in $\mathbb{A}^\delta$, which justifies this key step. 

The soundness of the rules (D), (T), (B) is verified in a similar way.

\subsection{Syntactic completeness}
\label{sec:SyntacticCompleteness}

In the present section, we show that all the axioms and rules of the basic logic are derivable in $\mathbf{A.L}$, and that
for any  $\Sigma\subseteq \{$(T), (4), (B), (D), (C)$\}$, the axioms and rules of  $\mathbf{L}\Sigma$ are derivable in  $\mathbf{A.L}\Sigma$.\footnote{That is, we show that, for any $\mathcal{L}$-axiom $s = A\vdash B \in \{$(T), (4), (B), (D), (C)$\}$, a derivation exists in  $\mathbf{A.L}\{s\}$ of the sequent $\nomj\leq A\vdash \nomj\leq B$, or equivalently of $B\leq \cnomm\vdash A\leq \cnomm$.}

The sequents $p \vd p$, $\bot \vd p$, $p \vd \top$, $p \vd p \vee q$ ($q \vd p \vee q$), and $p \wedge q \vd p$ ($q \wedge p \vd q$) are trivially derivable with one single application of the rule $Id_{\nomj\leq q}$, $\bot_w$, $\top_w$, $\vee_S$, and $\wedge_P$, respectively. The derivability of the rules $\frac{\varphi \vd \chi\ \chi \vd \psi}{\varphi \vd \psi}$, $\frac{\varphi \vd \chi\ \psi \vd \chi}{\varphi \vee \psi \vd \chi}$, and $\frac{\chi \vd \varphi\ \chi \vd \psi}{\chi \vd \varphi \wedge \psi}$ can be shown by  derivations in which  the cut rules, $\vee_P$, and $\wedge_S$, respectively, are applied. The two derivations below show the rules concerning the connectives $\wbox$ and $\wdia$:

\begin{center}
{\small
\begin{tabular}{rl}
\AXC{$\psi \leq \cnomm \vd \varphi \leq \cnomm$}
\LL{$\wbox_P$}
\UIC{$\nomj \leq \wbox \varphi, \psi \leq \cnomm \vd \nomj \leq \wbox \cnomm$}
\LL{$\wbox_S$}
\UIC{$\nomj \leq \wbox \varphi \vd\nomj \leq\wbox \psi$}
\DP
 \ & \  
\AXC{$\nomj \leq \varphi \vd \nomj \leq \psi$}
\RL{$\wdia_S$}
\UIC{$\wdia \psi \leq \cnomm, \nomj \leq \varphi \vd \wdia \nomj \leq \cnomm$}
\RL{$\wdia_P$}
\UIC{$\wdia \psi \leq \cnomm \vd \wdia \varphi \leq \cnomm$}
\DP
 \\
\end{tabular}
}
\end{center}

To show the admissibility of the substitution rule $\frac{\varphi \vd \psi}{\varphi(\chi/p) \vd \psi(\chi/p)}$, a straightforward induction on the derivation height of $\varphi \vd \psi$ suffices.

As to the axioms and rules of the basic logic $\mathbf{L}$,  below, we derive the axioms encoding the distributivity of  $\wdia$ over $\aor$ and $\bot$ in $\mathbf{A.L}$. The distributivity of $\wbox$ over $\wedge$ and $\top$ is derived similarly.

{\scriptsize
\begin{center}
\begin{tabular}{cc}
\negthinspace\negthinspace
\negthinspace\negthinspace
\negthinspace\negthinspace
\negthinspace\negthinspace
\negthinspace\negthinspace
\negthinspace\negthinspace
\negthinspace\negthinspace
\negthinspace\negthinspace
\bottomAlignProof
\AXC{\phantom{A}}
\LL{Id$_{\,\nomi \leq A}$}
\UIC{$\nomi \leq A \vd \nomi \leq A$}
\RL{$\wdia_S$}
\UIC{$\nomi \leq A, \wdia A \leq \cnomm \vd  \wdia \nomi \leq \cnomm$}
\RL{$\vee_S$}
\UIC{$\nomi \leq A, \wdia A \vee \wdia B \leq \cnomm \vd  \wdia \nomi \leq \cnomm$}
\RL{$\wdia \dashv \bbox$}
\UIC{$\nomi \leq A, \wdia A \vee \wdia B \leq \cnomm  \vd  \nomi \leq \blacksquare \cnomm$}
\LL{$S$}
\UIC{$\wdia A \vee \wdia B \leq \cnomm, \blacksquare \cnomm \leq \cnomn \vd   A  \leq \cnomn$}
\AXC{\phantom{B}}
\LL{Id$_{\,\nomi \leq B}$}
\UIC{$\nomi \leq B \vd \nomi \leq B$}
\RL{$\wdia_S$}
\UIC{$\nomi \leq B, \wdia B \leq \cnomm \vd  \wdia \nomi \leq \cnomm$}
\RL{$\vee_S$}
\UIC{$\nomi \leq B, \wdia A \vee \wdia B \leq \cnomm \vd  \wdia \nomi \leq \cnomm$}
\RL{$\wdia \dashv \bbox$}
\UIC{$\nomi \leq B, \wdia A \vee \wdia B \leq \cnomm \vd \nomi \leq \blacksquare \cnomm$}
\RL{$S$}
\UIC{$\wdia A \vee \wdia B \leq \cnomm,  \blacksquare \cnomm \leq \cnomn \vd   B \leq \cnomn$}
\LL{$\vee_P$}
\BIC{$\wdia A \vee \wdia B \leq \cnomm,  \blacksquare \cnomm \leq \cnomn \vd   A \vee B \leq \cnomn$}
\RL{$S$}
\UIC{$\nomh \leq A \vee B, \wdia A \vee \wdia B \leq \cnomm \vd \nomh \leq \blacksquare \cnomm$}
\LL{$\wdia \dashv \bbox^{-1}$}
\UIC{$\nomh \leq A \vee B, \wdia A \vee \wdia B \leq \cnomm \vd \wdia \nomh \leq \cnomm$}
\LL{$\wdia_P$}
\UIC{$\wdia A \vee \wdia B \leq \cnomm \vd \wdia(A \vee B) \leq \cnomm$}
\DP
 & 
\bottomAlignProof
\AXC{$\phantom{A}$}
\LL{Id$_{\,\nomj \leq \abot}$}
\UIC{$\nomj \leq \abot \vd \nomj \leq \abot$}
\RL{$\wdia\abot$}
\UIC{$\abot \leq \cnomm, \nomj \leq \abot \vd \nomj \leq \bbox \cnomm$}
\LL{$\wdia \dashv \bbox^{-1}$}
\UIC{$\abot \leq \cnomm, \nomj \leq \wdia \abot \vd \nomj \leq \cnomm$}
\LL{$\wdia_P$}
\UIC{$\abot \leq \cnomm \vd \wdia \abot \leq \cnomm$}
\DP
 \\
\end{tabular}
\end{center}
}

The syntactic completeness for the other axioms and rules of $\mathbf{L}$ can be shown in a similar way. In particular, the admissibility of the substitution rule can be proved by induction in a standard manner.

As to the axiomatic extensions of $\mathbf{L}$, Let us consider the axiom $(4)\  \wdia \wdia A \vdash \wdia A$. Using ALBA we generate the first order correspondent $\forall \nomj \forall \cnomm \, (\wdia \nomj \leq \cnomm \Rarr \wdia \wdia \nomj \leq \cnomm)$. Further processing the axiom $(4)$ using ALBA we obtain the equivalent first order correspondent (in the so-called `flat form'): $\forall \nomj \forall \nomh \forall \cnomm \, (\wdia \nomj \leq \cnomm \Rarr (\nomh \leq \wdia \nomj \Rarr \wdia \nomh \leq \cnomm))$, which can be written as a structural rule in the language of ALBA labelled calculi as follows:
\begin{center}
{\small
\AXC{$\Gamma \vd \wdia \nomj \leq \cnomm, \Delta$}
\LL{\fns $4$}
\UIC{$\Gamma, \nomh \leq \wdia \nomj \vd \wdia \nomh \leq \cnomm, \Delta$}
\DP
}
\end{center}
We now provide a derivation of the axiom $(4)$ in the basic labelled calculus $\mathbf{A.L}$ expanded with the previous structural rule $4$.
{\small 
\begin{center}
\AXC{ \ }
\LL{\fns Id$_A$}
\UIC{$\nomj \leq A \vd \nomj \leq A$}
\RL{\fns $\wdia_S$}
\UIC{$\nomj \leq A, \wdia A \leq \cnomm \vd \wdia \nomj \leq \cnomm$}
\LL{\fns $4$}
\UIC{$\nomj \leq A, \wdia A \leq \cnomm, \nomh \leq \wdia \nomj \vd \wdia \nomh \leq \cnomm$}
\RL{\fns $\wdia \dashv \bbox$}
\UIC{$\nomj \leq A, \wdia A \leq \cnomm, \nomh \leq \wdia \nomj \vd \nomh \leq \bbox \cnomm$}
\RL{\fns S}
\UIC{$\nomj \leq A, \wdia A \leq \cnomm, \bbox \cnomm \leq  \cnomm' \vd \wdia \nomj \leq \cnomm'$}
\LL{\fns $\wdia_P$}
\UIC{$\wdia A \leq \cnomm, \bbox \cnomm \leq  \cnomm' \vd \wdia A \leq \cnomm'$}
\RL{\fns S}
\UIC{$\nomj' \leq \wdia A, \wdia A \leq \cnomm \vd \nomj' \leq \bbox \cnomm$}
\LL{\fns $\wdia \dashv \bbox^{-1}$}
\UIC{$\nomj' \leq \wdia A, \wdia A \leq \cnomm \vd \wdia \nomj' \leq \cnomm$}
\LL{\fns $\wdia_P$}
\UIC{$\wdia A \leq \cnomm \vd \wdia \wdia A \leq \cnomm$}
\DP
\end{center}
 }



Analogously, we provide a derivation of the axioms (T), (B), (D), and (C) in the basic labelled calculus $\mathbf{A.L}$ expanded with the structural rules $T$, $B$, $D$, and $C$, respectively.
\begin{center}
{\small
\begin{tabular}{rl}
\bottomAlignProof
\AXC{ \ }
\LL{\fns Id$_A$}
\UIC{$ A \leq \cnomm \vd A \leq \cnomm$}
\RL{\fns $\wbox_P$}
\UIC{$\nomj \leq \wbox A, A \leq \cnomm \vd \nomj \leq \wbox \cnomm$}
\RL{\fns $T$}
\UIC{$\nomj \leq \wbox A, A \leq \cnomm \vd \nomj \leq \cnomm$}
\RL{\fns $S$}
\UIC{$\nomj \leq \wbox A \vd \nomj \leq A$}
\DP
&
\bottomAlignProof
\AXC{ \ }
\RL{\fns Id$_A$}
\UIC{$\nomj \leq A \vd \nomj \leq A$}
\RL{\fns $\wdia_S$}
\UIC{$\nomj \leq A, \wdia A \leq \cnomm \vd \wdia \nomj \leq \cnomm$}
\RL{\fns $B$}
\UIC{$\nomj \leq A, \wdia A \leq \cnomm \vd \nomj \leq \wbox \cnomm$}
\LL{\fns $\wbox_S$}
\UIC{$\nomj \leq A \vd \nomj \leq \wbox \wdia A$}
\DP
 \\ \\
\bottomAlignProof
\AXC{ \ }
\LL{\fns Id$_A$}
\UIC{$A \leq \cnomn \vd A \leq \cnomn$}
\RL{\fns $\wbox_P$}
\UIC{$\nomj \leq \wbox A, A \leq \cnomn \vd \nomj \leq \wbox \cnomn$}
\RL{\fns $\bdia \dashv \wbox$}
\UIC{$\nomj \leq \wbox A, A \leq \cnomn \vd \bdia \nomj \leq \cnomn$}
\LL{\fns S}
\UIC{$\nomk \leq \bdia \nomj, \nomj \leq \wbox A \vd \nomk \leq A$}
\RL{\fns $\wdia_S$}
\UIC{$\nomk \leq \bdia \nomj, \nomj \leq \wbox A, \wdia A \leq \cnomm \vd \wdia \nomk \leq \cnomm$}
\LL{\fns $D$}
\UIC{$\nomj \leq \wbox A, \wdia A \leq \cnomm \vd \nomj \leq \cnomm$}
\LL{\fns S}
\UIC{$\wdia A \leq \cnomm \vd \wbox A \leq \cnomm$}
\DP
 & 
\bottomAlignProof
\AXC{ \ }
\LL{\fns Id$_A$}
\UIC{$A \leq \cnomo \vd A \leq \cnomo$}
\LL{\fns $\wbox_P$}
\UIC{$\nomj \leq \wbox A, A \leq \cnomo \vd \nomj \leq \wbox \cnomo$}
\LL{\fns $\bdia \dashv \wbox$}
\UIC{$\nomj \leq \wbox A, A \leq \cnomo \vd \bdia \nomj \leq \cnomo$}
\RL{\fns $S$}
\UIC{$\nomk \leq \bdia \nomj, \nomj \leq \wbox A \vd \nomk \leq A$}
\RL{\fns $\wdia_S$}
\UIC{$\nomk \leq \bdia \nomj, \nomj \leq \wbox A, \wdia A \leq \cnomm \vd \wdia \nomk \leq \cnomm$}
\LL{\fns $C$}
\UIC{$\nomh \leq \wdia \nomj, \nomj \leq \wbox A, \wdia A \leq \cnomm \vd \bdia \nomh \leq \cnomm$}
\RL{\fns $\bdia \dashv \wbox^{-1}$}
\UIC{$\nomh \leq \wdia \nomj, \nomj \leq \wbox A, \wdia A \leq \cnomm \vd \nomh \leq \wbox \cnomm$}
\RL{\fns $S$}
\UIC{$\nomj \leq \wbox A, \wdia A \leq \cnomm, \wbox \cnomm \leq \cnomn \vd \wdia \nomj \leq \cnomn$}
\LL{\fns $\wdia_P$}
\UIC{$\wdia A \leq \cnomm, \wbox \cnomm \leq \cnomn \vd \wdia \wbox A \leq \cnomn$}
\RL{\fns $S$}
\UIC{$\nomi \leq \wdia \wbox A, \wdia A \leq \cnomm \vd \nomi \leq \wbox \cnomm$}
\RL{\fns $\wbox_S$}
\UIC{$\nomi \leq \wdia \wbox A \vd \nomi \leq \wbox \wdia A$}
\DP
\end{tabular}
}
\end{center}

\subsection{Conservativity}

In the present section, we argue that, for any $\Sigma\subseteq \{(T), (4), (B), (D), (C)\}$ and all $\mathcal{L}$-formulas $A$ and $B$, if the $\mathbf{A.L}$-sequent $\nomj\leq A\vd \nomj\leq B$ is derivable in $\mathbf{A.L}\Sigma$, then the $\mathcal{L}$-sequent $A\vd B$ is an $\mathbf{L}.\Sigma$-theorem.

Indeed, because the rules of $\mathbf{A.L}$ are sound in the class $\mathsf{K}^\delta (\Sigma)= \{\mathbb{A}^\delta\mid \mathbb{A}\models \Sigma\}$, the assumption implies that $\mathbb{A}^\delta\models \forall \ol{p}\forall \nomj(\nomj\leq A\Rarr \nomj\leq B)$ which, by complete join-generation, is equivalent to 
$\mathbb{A}^\delta\models \forall \ol{p}( A\leq B)$. Since, as discussed in Section \ref{ssec:interpretation ALBA}, $\mathbf{L}.\Sigma$ is complete w.r.t.~$\mathsf{K}^\delta(\Sigma)$, this implies that $A\vd B$ is a theorem of $\mathbf{L}.\Sigma$, as required.

\subsection{Cut elimination and subformula property}
\label{CutEliminationAndSubformulaProperty}

As usual in the tradition of display calculi, we first characterize a new class of calculi. The actual definition is given in Appendix \ref{sec:ProperAlgebraicLabelledCalculi}. Here we just mention that a {\em proper labelled calculus} is a proof systems satisfying the conditions C$_2$-C$_8$ of Definition \ref{def:ProperLabelledCalculi}. Then, in the appendix, we prove the following general result, namely a canonical cut elimination theorem {\em \`a la} Belnap for any calculus in this class:

\begin{theorem}
\label{thm:MetaCutElimination}
Any proper labelled calculus enjoys cut-elimination. If also the condition C$_1$ in Definition \ref{def:ProperLabelledCalculi} is satisfied, then the proof system enjoys the subformula property.
\end{theorem}

Finally, we obtain that the calculi $\mathbf{A.L}\Sigma$ introduced in Section \ref{sec:LabelledCalculusAL} enjoys cut elimination and subformula property thanks to the following result, proved in Section \ref{ALIsAProperLabelledCalculus}.

\begin{corollary}
\label{thm:ALIsProper}
$\mathbf{A.L}\Sigma$ is a proper labelled calculus.
\end{corollary}


\section{$\mathbf{A.L}\Sigma$ is a proper labelled calculus}
\label{ALIsAProperLabelledCalculus}

In this section we first show that $\mathbf{A.L}\Sigma$ has the display property and then we provide a proof of Corollary \ref{thm:ALIsProper} as stated in Section \ref{CutEliminationAndSubformulaProperty}.

\begin{lemma} \label{lem:exacttwo}
Let $s$ be any derivable sequent in $\mathbf{A.L}\Sigma$. Then, up to renaming of the variables, every nominal or conominal occurring in $s$ occurs in it exactly twice and with opposite polarity.
\end{lemma}
\begin{proof}
The proof is by induction on the height of the derivation. The statement is trivially  true for the axioms. All the rules in the calculi with the exception of the cut-rules either (a) introduce two occurrences of a new label,  (b) eliminate two occurrences of the same existing label, or (c) keep the labels in the sequent intact.  Before applying the same reasoning to Cut$_{\,\nomj\leq A}$ and Cut$_{\,A\leq\cnomm}$, we must take care of renaming nominals and conominals in $\Gamma,\Gamma',\Delta,\Delta'$ to avoid conflicts. Thus, up to substitution, any variable occurring in a derivable sequent occurs exactly twice.
\end{proof}
To prove that $\mathbf{A.L}\Sigma$ has the display property (see Definition \ref{def:Display}) we need the following definition:

\begin{definition}
\label{def:twin}
Let $s= \Gamma \vd \Delta$  be any sequent. For any structure $\sigma= \nomj \leq \NCT$,  (resp.~$\sigma=\NCT \leq \cnomm$) in $\Gamma$, we say that it has a {\em $\nomj$-twin} (resp.~{\em $\cnomm$-twin}) iff there exists  exactly one occurrence of $\nomj$ (resp.~$\cnomm$) in $\Delta$. 
\end{definition}

\begin{proposition} \label{prop:twin property}
For any derivable sequent $s=\Gamma \vd \Delta$ and  any structure $\sigma = \nomj \leq \NCT$ (resp.~$\sigma = \NCT \leq \cnomm$)  in $\Gamma$  there exists a sequent $s' =\Gamma' \vd \Delta'$ which is interderivable  with $s$ such that $ \sigma \in \Gamma'$ and  $\sigma$ has a $\nomj$-twin (resp.~$ \cnomm$-twin) in $s'$. 
\end{proposition}

\begin{proof}
A sequent $s$ is said to have the {\em $\nomj$-twin property} (resp.\ $\cnomm$-twin property) iff every structure of the form $\nomj \leq \NCT$ (resp.~$\NCT \leq \cnomm$) has a $\nomj$-twin (resp.~$\cnomm$-twin) in $\Delta$. Notice that the conclusions of initial rules satisfy all twin properties trivially. We say that an inference rule preserves the twin property if the premises having a certain twin property implies that the conclusion has the same twin property. 

All the rules in the basic calculus which do not involve switching of  structural terms (i.e.~all the rules except for the rules S$\NCT\cnomm$, S$\cnomm\NCT$,  S$\NCT\nomj$, S$\nomj\NCT$,  S$\NCT\NCT'\cnomm$,  S$\nomj\NCT\NCT'$) and the rules $T$, $B$, and $C$ preserve the twin property, since, for each nominal $\nomj$ or conominal $\cnomm$ in the rule, one of the following holds:
\begin{enumerate}
    \item each occurrence of $\nomj$ or $\cnomm$ remains on the same side in the conclusion as it was in the premise (adjunction rules, rules for propositional connectives, $T$, $B$);
    \item exactly two occurrences of $\nomj$ (resp.~$\cnomm$) are eliminated (cut,  S$\cnomm$, S$\nomj$, $\wbox_S$, $\wdia_P$, $D$, $C$);\footnote{Given the result in Lemma \ref{lem:exacttwo}, notice that this is the same as saying that {\em all} occurrences of $\nomj$ (resp.~$\cnomm$) are eliminated, except in the case of cut rules, which have two premises. But, if the cut formula has a $\nomj$-twin (resp.\ $\cnomm$-twin) in $\Gamma$ in the one premise and in $\Delta'$ in the other, the conclusion  $\Gamma,\Gamma'\vd\Delta,\Delta'$ will still have the $\nomj$-twin (resp.\ $\cnomm$-twin) property.}
    \item $\nomj$ (resp.~$\cnomm$) does not occur in the premise, and exactly two occurrences of $\nomj$ (resp.~$\cnomm$) are added in the conclusion, one in the antecedent and one in the consequent of the conclusion ($\top_\wbox$,$\bot_\wdia$,$\wbox_P$,$\wdia_S$, $C$).
\end{enumerate}
Notice that rule $C$ occurs both in items 2 and 3 and, in particular, we can assume that the nominal $\nomh$ in the conclusion of $C$ is fresh (see the proof of Lemma \ref{thm:PreservationOfPrincipalFormulasApproximants}).  

In the case one of the S$\NCT\cnomm$, S$\cnomm\NCT$,  S$\NCT\nomj$, S$\nomj\NCT$,  S$\NCT\NCT'\cnomm$,  S$\nomj\NCT\NCT'$  rules is applied, the twin property can be broken, as nominals (resp.~conominal) contained in the structural terms $\NCT,\NCT'$ switch side without their twin nominals (resp.~conominals) doing the same. Let  $R=\frac{\Gamma_1 \vd \Delta_1}{\Gamma_2 \vd \Delta_2}$ be an application of any of these rules.  Let $\sigma \in \Gamma_1 \cap \Gamma_2 $ be a structure which has a twin in $\Delta_1$ but not in $\Delta_2$.  This is only possible if $\sigma$ is of the form $\nomi \leq \NCT_1$  (resp.~$\NCT_1 \leq \cnomn$) and there exists a $\sigma'$ of the form $\NCT \leq \cnomm$ (resp.~ $\nomj \leq \NCT$) for some $\NCT$ containing $\nomi$ (resp.~$\cnomn$) which will appear in $\Gamma_2$ with $\nomj$ (resp.~$\cnomm$) introduced fresh. We can switch the term $\NCT$ back to the right by using the switch rule applicable to the conclusion of $R$, which gives a sequent of the same form as $\Gamma_1\vd\Delta_1$, but with a fresh conominal (resp.\ nominal) in the place of $\cnomm$ (resp.\ $\nomj$). Here, we show the proof for the rules S$\NCT\NCT'\cnomm$ and S$\cnomm\NCT$, the proof for other rules being similar.
\begin{center} 
{\small
\begin{tabular}{rl}
\AXC{$\Gamma,\nomj\leq\NCT'\vd\nomj\leq\NCT,\Delta$}
\LL{S$\NCT\NCT'\cnomm$}
\UIC{$\Gamma,\NCT\leq\cnomm\vd\NCT'\leq\cnomm,\Delta$}
\LL{S$\nomj\NCT\NCT'$}
\UIC{$\Gamma,\nomj'\leq\NCT'\vd\nomj'\leq\NCT,\Delta$}
\DP
 \ & \  
\AXC{$\Gamma,A\leq\cnomm\vd\NCT\leq\cnomm,\Delta$}
\RL{S$\NCT\nomj$}
\UIC{$\Gamma,\nomj\leq\NCT\vd \nomj\leq A,\Delta$}
\RL{S$\cnomm\NCT$}
\UIC{$\Gamma,A\leq\cnomm'\vd\NCT\leq\cnomm',\Delta$}
\DP
 \\
\end{tabular}
}
\end{center}
Let $s' =\Gamma_3 \vd \Delta_3$ be any sequent derived from $\Gamma_2 \vd \Delta_2$ such that $\sigma\in \Gamma_3$. If $\sigma' \notin \Gamma_3$, it must have been switched to the right by a switch rule, as all other rules leave pure structures on the left intact. However, whenever a switch rule is applied, the term $\NCT$ occurs on the right of the sequent and has a twin. Thus, if the last rule of the derivation is such a switch rule, then $\sigma$ has a twin in $s'$.  If $\sigma' \in \Gamma_3$ and $\sigma'$ has a twin structure, then we can switch $\NCT$ to right by the  application of an appropriate switch rule leading to a interderivable sequent in which $\sigma$ has a twin. Thus, reasoning by induction, it is enough to show that $\sigma'$ has a twin when it occurs for the first time (by the application of the switch rule). This is immediate from the fact that the label occurring in $\sigma' \in \Gamma_3$ in any switch rule is of the form $\nomj \leq \NCT$ or $\NCT \leq \cnomm $ where $\nomj$ or $\cnomm$ is fresh and also occurs in $\Delta_3$. Thus switch rules preserve the twin property.

Let  $R=\frac{\Gamma_1 \vd \Delta_1}{\Gamma_2 \vd \Delta_2}$ be an application of rule $(4)$.  Let $\sigma \in \Gamma_1 \cap \Gamma_2 $ be a structure which has a twin in $\Delta_1$ but not in $\Delta_2$. This is only possible if $\sigma$ is of the form $\nomj \leq \NCT \in \Gamma$ for some structure $\NCT$. In this case, after the application of $4$ we can reintroduce the relevant twin structure $\wdia \nomj \leq \cnomm$ in the following way.
\begin{center}
{\small 
\AXC{$\nomj \leq \NCT \in \Gamma \vd \wdia \nomj \leq \cnomm$}
\LL{\fns $4$}
\UIC{$\nomj \leq \NCT \in \Gamma, \nomh \leq \wdia \nomj \vd \wdia \nomh \leq \cnomm$} 
\RL{\fns $\wdia \dashv \bbox$}
\UIC{$\nomj \leq \NCT \in \Gamma, \nomh \leq \wdia \nomj \vd \nomh \leq \bbox\cnomm$}
\LL{\fns S$\NCT\NCT'\cnomn$}
\UIC{$\nomj \leq \NCT \in \Gamma,\bbox\cnomm \leq \cnomn  \vd \wdia \nomj  \leq\cnomn$}
\DP
 }
\end{center}
The rest of the proof now follows by the same inductive algorithm.

\end{proof}
We can now prove the display property of the calculus. 
\begin{proposition}
\label{thm:DisplayPropertyAL}
If $s =\Gamma \vd \Delta $ is a derivable $\mathbf{A.L}\Sigma$-sequent, then every structure $\sigma$ occurring in $s$ is {\em displayable} (see Definition \ref{def:Display}).
\end{proposition}

\begin{proof}\label{proof:thm:DisplayPropertyAL}
We prove the display property only for the structures of the form $\nomj \leq \NCT$ and $\nomj \leq A$. The proof for structures of the form $ \NCT \leq \cnomm$ and $A \leq \cnomm$ are dual. Let  $\sigma$ be a structure of the form  $\nomj \leq \NCT$ or $\nomj \leq A$.  By Lemma \ref{lem:exacttwo}, $\nomj$  occurs exactly  twice in any derivable sequent. Let $\sigma'$ be the structure containing the other occurrence of $\nomj$. If $\sigma'$ is a labelled formula we are done. The set of well-formed pure structures  containing $\nomj$ is the following: $\{\nomj \leq \cnomm, \wdia\nomj\leq\cnomm, \bdia \nomj\leq \cnomm, \nomj\leq \wbox \cnomm, \nomj\leq\bbox \cnomm, \nomi\leq\wdia\nomj, \nomi\leq\bdia\nomj\}$.
For any derivable sequent $s = \Gamma \vdash \Delta$, it is easy to verify the following (by inspection of all the rules in $\mathbf{A.L}$): 
\begin{itemize}
\item[(a)] if $\wdia\nomj\leq\cnomn$ (resp.~$\nomi\leq \wbox \cnomm$) occurs in $s$, then it occurs in $\Delta$, given that $\wdia_S$ and $\bot_\wdia$ (resp.~$\wbox_P$ and $\top_\wbox$) are the only rules introducing such a structure;
\item[(b)] if $\bdia \nomj \leq \cnomn$ (resp.~$\nomi \leq \bbox \cnomm$) occurs in $s$, then it occurs in $\Delta$, given that $\bdia$ and $\bbox$ can be introduced only via adjunction rules to structures  $\wdia\nomj\leq\cnomn$ and $\nomi\leq \wbox \cnomm$ which only occur in $\Delta$ by (a);
\item[(c)] if any pure structure $t \in\{\nomi\leq\wdia\nomj, \nomi\leq\bdia\nomj, \wbox\cnomm\leq\cnomn, \bbox\cnomm\leq\cnomn\}$ occurs in $s$, then it occurs in $\Gamma$, given that they can only be introduced via a switch rule (namely, S$\NCT\cnomm$, S$\NCT\nomj$, S$\NCT\NCT'\cnomm$, S$\nomj,\NCT\NCT'$).
\end{itemize}

Case~(a) and (b). If $\nomj$ (resp.\ $\cnomm$) occurs in some pure structure $\wdia\nomj\leq\cnomn$ or $\bdia \nomj\leq \cnomn$ (resp.\ $\nomi\leq \wbox \cnomm$ or $\nomi\leq\bbox \cnomm$), then it occurs in $\Delta$ and it is displayable through a single application of an adjunction rule.

Case~(c). If $\nomj$ (resp.~$\cnomm$) occurs in some pure structure $\nomi\leq\wdia\nomj$, $\nomi\leq\bdia\nomj$ (resp.~$\wbox\cnomm\leq\cnomn$, $\bbox\cnomm\leq\cnomn$), then it occurs in $\Gamma$. If we can apply a switch rule moving the relevant complex term to $\Delta$, we reduce ourselves to the previous cases (a) or (b) and we are done. We will consider the case where a pure structure $\nomi\leq\wdia\nomj$ occurs on the left. All the other cases are treated analogously.  By Proposition \ref{prop:twin property}, there exists a sequence $s' =\Gamma' \vd \Delta'$ such that $s'$ is interderivable with $s$,  $\nomi\leq\wdia\nomj$ occurs in $\Gamma'$ and has an $\nomi$-twin (resp.~$\cnomn$-twin) in $s'$. Thus, $\nomi$ occurs in $\Delta'$ by definition. 
By (a) and (b), $\nomi$ can occur in $\Delta'$ in one of the structures $\nomi \leq A$, $\nomi \leq \cnomm$  $\wdia \nomi \leq \cnomm$, or $\bdia \nomi \leq \cnomm$. If $\nomi$ occurs in $\nomi \leq A$ or  $\nomi \leq \cnomm$, we can apply a switch rule with $\nomi \leq \wdia\nomj$ to reduce ourselves to previous cases. If $\nomi$ occurs in the structure of the form $\wdia \nomi \leq \cnomm$ or $\bdia \nomi \leq \cnomm$, we first apply adjunction and then proceed in the same way. This concludes the proof of display property for the basic calculus. 

In case of the axiomatic extensions, we need to argue that the added rules preserve the property in Proposition \ref{prop:twin property}. Notice that all the additional rules preserve the Lemma \ref{lem:exacttwo}. 

In case of rules {\em (T)}, {\em (B)}, and {\em (C)} no variables nominals or conominals switch side thus the twin structures are preserved by these rules. 
{\em (4)}: The rule 4 moves $\wdia \nomj$ from the right to the left. Thus, the only structure in $\Gamma$ which loses the twin property is some  structure $\sigma$ in $\Gamma$ containing  $\nomj$. However, just like in the case of the switch rules we can switch $\nomj$ back to the right by applying  an adjunction followed by a switch rule on $\nomh$ as follows.
\begin{center} 
{\small
\AXC{$\Gamma,  \nomh \leq \wdia\nomj \vd \wdia \nomh \leq \cnomm $}
\LL{\fns $\wdia \dashv \bbox$}
\UIC{$\Gamma,   \nomh \leq \wdia\nomj \vd  \nomh \leq \bbox\cnomm $}
\LL{S$\NCT\NCT'\cnomm$}
\UIC{$\Gamma,   \bbox\cnomm  \leq  \cnomn  \vd  \wdia\nomj  \leq \cnomn$}
\DP
}
\end{center}
Thus, the twin of $\sigma$ reappears. We can now argue that the existence of an interderivable sequent with a twin is preserved in a derivation containing an application of this rule in similar way as we did with the switch rules.

{\em (D): } In case of rule {\em D}, as the premise contains the structure $\nomk \leq \bdia \nomj$, there must be another structure $\sigma \in \Gamma$ containing $\nomj$ of the form $\nomj \leq A$ or $\nomj \leq \NCT$. The premise must be interderivable with a sequent in which $\sigma$ has a twin. However, in the conclusion $\nomk$ does not occur and $\nomj \leq \wbox \cnomm$ occurs on the right. Thus, $\sigma$ has a twin in the conclusion. All other nominals and conominals in the premise stay on the same side in the conclusion. Thus, rule {\em D} preserves the property of existence of interederivable sequent with the twin property.

\end{proof}

To show that $\mathbf{A.L}\Sigma$ is a proper labelled calculus, we need to verify that $\mathbf{A.L}\Sigma$ satisfies each condition in Definition \ref{def:ProperLabelledCalculi}. The verification of the conditions C$6$ and C$7$ requires to preliminarily show the following:

\begin{lemma}[Preservation of principal formulas' approximants]
\label{thm:PreservationOfPrincipalFormulasApproximants} 
If the sequents $s$ and $s'$ occur in the same branch $b$ of an $\mathbf{A.L}\Sigma$-derivation $\pi$, 
$\nomj \leq A \in s$ and $\nomi \leq A \in s'$ (resp.~$A \leq \cnomm \in s$ and $A \leq \cnomn \in s'$) are {\em congruent} in $b$ (see Definition \ref{def:Display}),  $\nomj \leq A$ (resp.~$A \leq \cnomm$) is principal in $s$ and $\nomi \leq A$ (resp.~$A \leq \cnomn$) is in display in $s'$, then the term occurring in $\nomi \leq A$ (resp.~$A \leq \cnomn$) in $\pi$ can be renamed in a way such that $\nomi = \nomj$ (resp.~$\cnomn = \cnomm$), and the new derivation $\pi'$ is s.t.~$\pi' \equiv \pi$ modulo a renaming of some nominal or conominal.
\end{lemma}


\begin{proof}
Assume that $\nomj \leq A \in s$ and $\nomi \leq A \in s'$ (resp.~$A \leq \cnomm \in s$ and $A \leq \cnomn \in s'$) are congruent in the branch $b$ of a derivation $\pi$, $\nomj \leq A$ (resp.~$A \leq \cnomm$) is principal in $s$ and $\nomi \leq A$ (resp.~$A \leq \cnomn$) is in display in $s'$. This means that there is a sub-branch $b' \subseteq b$ connecting $s$ and $s'$ and the height of $s$ is strictly smaller than the height of $s'$ in $\pi$, given that $\nomj \leq A$ (resp.~$A \leq \cnomm$) is principal in $s$. If $\nomj \leq A$ stays parametric in $b'$, then $\nomj = \nomi$ (resp.~$\cnomm = \cnomn$). If not, it means that $\nomj \leq A$ (resp.~$A \leq \cnomm$) is nonparametric in an even number of applications of switch rules occurring in $b'$, given that $\nomj \leq A$ and $\nomi \leq A$ (resp.~$A \leq \cnomm$ and $A \leq \cnomn$) are both approximated by a nominal (resp.~conominal) and a single application of a switch rules changes the nominal (resp.~conominal) approximating a formula into a conominal (resp.~nominal). W.l.o.g.~we can confine ourselves to consider a branch $b'$ with exactly two applications of switch rule involving $\nomj \leq A$ and $\nomi \leq A$ (resp.~$\cnomm \leq A$) as nonparametric structures. If  a rule $R$ is applied in $b'$, $R$ is not a switch rule, and $x$ is a nominal or conominal occurring in a nonparametric strucure in the conclusion of $R$ but not necessarily occurring in the premise(s) of $R$, then we can pick $x$ fresh in the entire sub-branch $b'' \subseteq b'$ connecting the conclusion of $R$ with the sequent $s$. In the case of $\mathbf{A.L}\Sigma$, such rule $R$ is either $\wbox_P$, $\wdia_S$, $C$, or $4$. Therefore, $\nomj$ (resp.~$\cnomm$) occurs neither in the premise nor in any parametric structure in the conclusion of $S$, where $S$ is any application of a switch rules in the branch $b'$ and s.t.~$\nomi \leq A$ (resp.~$A \leq \cnomn$) is a nonparametric structure in the conclusion of $S$. So, the side conditions of switch rules are satisfied, and in any application of such $S$ we can switch $\nomi$ for $\nomj$ (resp.~$\cnomn$ for  $\cnomm$).
\end{proof}


We now proceed to check all the conditions C$_1$-C$_8$ defining proper labelled calculi.

\begin{proof}
The fact that $\mathbf{A.L}\Sigma$ satisfies condition C$'_5$ is proved in Proposition \ref{thm:DisplayPropertyAL}. To show that $\mathbf{A.L}\Sigma$ satisfies condition C$_8$ we need to consider all possible principal cuts.

Below we consider cut formulas introduced by initial rules and we exhibit a new cut-free proof that is an axiom as well (notice that axioms are defined with empty contexts). The principal cut formula is $\nomj \leq \aatop$ or $\nomj \leq p$, $R \in \{\textrm{Id}_{\nomj \leq p}, \textrm{Id}_\aatop, \aatop_w\}$, and $x,y,z$ are instantiated accordingly to $R$ (the proof for $\abot \leq \cnomm$ or $p \leq \cnomm$, and $R \in \{\textrm{Id}_{p\leq m}, \textrm{Id}_\abot, \abot_w\}$ is similar and it is omitted). 


{\small
\begin{center}
\begin{tabular}{ccc}
\bottomAlignProof
\AXC{$\pi_1$}
\RL{\fns $R$}
\UIC{$\nomj \leq x \vd \nomj \leq y$}
\AXC{$\pi_2$}
\RL{\fns $R$}
\UIC{$\nomj \leq y \vd  \nomj \leq z$}
\LL{\fns Cut$_{\,\nomj \leq \aatop}$}
\BIC{$\nomj \leq x \vd \nomj \leq z$}
\DP
&
\bottomAlignProof
$\rightsquigarrow$
&
\ \ 
\bottomAlignProof
\AXC{$\pi_i$}
\RL{\fns $R$}
\UIC{$\nomj \leq x \vd \nomj \leq z$}
\DP 
 \\
\end{tabular}
\end{center}
 }

Below we consider formulas introduced by logical rules. The side conditions on logical rules impose that the parametric structures are in display or can be displayed with a single application of an adjunction rule, therefore we can provide the following proof transformations where the newly generated cuts are well-defined (in particular the cut formulas are in display) and of lower complexity (the new cut formulas are immediate subformulas of the original cut formulas). 

The principal cut formula is $\wbox A$ and $A$ is 
$\cnomm$-labelled formula (see Definition \ref{def:j-m-LabelledFormulas}).

\vspace{-0.3cm}

{\small
\begin{center}
\begin{tabular}{ccc}
\negthinspace\negthinspace
\negthinspace\negthinspace
\negthinspace\negthinspace
\negthinspace\negthinspace
\negthinspace\negthinspace
\negthinspace\negthinspace
\negthinspace\negthinspace
\negthinspace\negthinspace
\negthinspace\negthinspace
\negthinspace\negthinspace
\negthinspace\negthinspace
\negthinspace\negthinspace
\bottomAlignProof
\AXC{\ \, $\vdots$ \raisebox{1mm}{$\pi_1$}}
\noLine
\UIC{$\Gamma, A \leq \cnomm \vd \nomj \leq \wbox \cnomm, \Delta$}
\RL{\fns $\wbox_S$}
\UIC{$\Gamma \vd \nomj \leq \wbox A, \Delta$}
\AXC{\!\!\!\!\!\!\!\!\!\!\!\!$\vdots$ \raisebox{1mm}{$\pi_2$}}
\noLine
\UIC{$\Gamma' \vd A \leq \cnomm, \Delta'$}
\LL{\fns $\wbox_P$}
\UIC{$\Gamma', \nomj \leq \wbox A \vd \nomj \leq \wbox \cnomm,\Delta'$}
\LL{\fns Cut$_{\,\nomj\leq \wbox A}$}
\BIC{$\Gamma,\Gamma' \vd \Delta, \nomj \leq \wbox \cnomm, \Delta'$}
\DP

 & 
\bottomAlignProof
$\rightsquigarrow$ 
 & 
\bottomAlignProof
\AXC{\!\!\!\!\!\!\!\!\!\!\!\!$\vdots$ \raisebox{1mm}{$\pi_2$}}
\noLine
\UIC{$\Gamma' \vd A \leq \cnomm, \Delta'$}
\AXC{\ \, $\vdots$ \raisebox{1mm}{$\pi_1$}}
\noLine
\UIC{$\Gamma, A \leq \cnomm \vd \nomj \leq \wbox \cnomm,\Delta$}
\RL{\fns $\bdia \dashv \wbox$}
\UIC{$\Gamma, A \leq \cnomm \vd \bdia \nomj \leq \cnomm,\Delta$}
\RL{\fns Cut$_{\,A \leq \cnomm}$}
\BIC{$\Gamma,\Gamma'\vd \Delta, \bdia \nomj \leq \cnomm, \Delta'$}
\LL{\fns $\bdia \dashv \wbox^{-1}$}
\UIC{$\Gamma,\Gamma'\vd \Delta, \nomj \leq \wbox \cnomm, \Delta'$}
\DP 
 \\
\end{tabular}
\end{center}
 }

The principal cut formula is $\wbox A$ and $A$ is a 
$\nomj$-labelled formula (see Definition \ref{def:j-m-LabelledFormulas}): in this case the original cut is as in the proof above. The principal cut elimination transformation requires we perform a new cut on the immediate labelled subformulas $A \leq \cnomm$, but, given $A$ is a $\nomj$-labelled formula by assumption, only Cut$_{\nomj \leq A}$ can be applied, so we need first to apply the appropriate switch rules changing the approximants of the immediate subformulas $A$ on both branches. In branch $\pi_1$ of the original proof, we can apply adjunction and derive $\Gamma, A \leq \cnomm \vd \bdia \nomj \leq \cnomm$ where $\bdia \nomj \leq \cnomm$ is the twin structure (see Definition \ref{def:twin}) of $A \leq \cnomm$, so we can apply the rule S$\NCT\nomj$. In branch $\pi_2$, Proposition \ref{prop:twin property} ensures that $\Gamma' \vd A \leq \cnomm, \Delta'$ is interderivable with $ \Gamma'', x \leq \cnomm \vd A \leq \cnomm, \Delta''$ for some $x$ formula or $x$ pure structure, where $x \leq \cnomm$ is the twin structure of $A \leq \cnomm$. Therefore, we can apply the appropriate switch rule changing the approximant of $A$. The proof transformation is detailed below.

{\small
\begin{center}
\begin{tabular}{ccc}

 & 
\bottomAlignProof
$\rightsquigarrow$ 
 & 
\bottomAlignProof
\AXC{\ \, $\vdots$ \raisebox{1mm}{$\pi_1$}}
\noLine
\UIC{$\Gamma, A \leq \cnomm \vd \nomj \leq \wbox \cnomm,\Delta$}
\RL{\fns $\bdia \dashv \wbox$}
\UIC{$\Gamma, A \leq \cnomm \vd \bdia \nomj \leq \cnomm,\Delta$}
\RL{\fns S$\NCT\nomj$}
\UIC{$\Gamma, \nomi \leq \bdia \nomj \vd \nomi \leq A,\Delta$}
\AXC{\!\!\!\!\!\!\!\!\!\!\!\!$\vdots$ \raisebox{1mm}{$\pi_2$}}
\noLine
\UIC{$\Gamma' \vd A \leq \cnomm, \Delta'$}
\dashedLine
\RL{\fns Prop.~\ref{prop:twin property}}
\UIC{$\Gamma'', x \leq \cnomm \vd A \leq \cnomm, \Delta''$}
\RL{\fns S}
\UIC{$\Gamma'', \nomi \leq A \vd \nomi \leq x, \Delta''$}
\RL{\fns Cut$_{\,\nomi \leq A}$}
\BIC{$\Gamma, \Gamma'', \nomi \leq \bdia \nomj \vd \nomi \leq x, \Delta, \Delta''$}
\RL{\fns S}
\UIC{$\Gamma, \Gamma'', x \leq \cnomm \vd \bdia \nomj \leq \cnomm, \Delta, \Delta''$}
\dashedLine
\RL{\fns Prop.~\ref{prop:twin property}}
\UIC{$\Gamma, \Gamma' \vd \bdia \nomj \leq \cnomm, \Delta, \Delta'$}
\LL{\fns $\bdia \dashv \wbox^{-1}$}
\UIC{$\Gamma, \Gamma' \vd \nomj \leq \wbox \cnomm, \Delta, \Delta'$}
\DP 
 \\
\end{tabular}
\end{center}
 }

The principal cut formula is $A \aand B$ and both immediate subformulas are $\nomj$-labelled formulas (the case in which at least one immediate subformula is an $\cnomm$-labelled formula is analogous to the proof transformation step for $\wdia A$ and it is omitted. The case for $A \aor B$ are similar and they are omitted).
{\small
\begin{center}
\begin{tabular}{ccc}
\negthinspace\negthinspace
\negthinspace\negthinspace
\negthinspace\negthinspace
\negthinspace\negthinspace
\negthinspace\negthinspace
\negthinspace\negthinspace
\negthinspace\negthinspace
\negthinspace\negthinspace
\negthinspace\negthinspace
\negthinspace\negthinspace
\negthinspace\negthinspace
\negthinspace\negthinspace
\bottomAlignProof
\AXC{\!\!\!\!\!\!\!\!\!\!$\!\vdots$ \raisebox{1mm}{$\pi_1$}}
\noLine
\UIC{$\Gamma\vd \nomj\leq A_1,\Delta$}
\AXC{\!\!\!\!\!\!\!\!\!\!\!$\vdots$ \raisebox{1mm}{$\pi_2$}}
\noLine
\UIC{$\Gamma\vd \nomj\leq A_2,\Delta$}
\RL{\fns $\land_S$}
\BIC{$\Gamma \vd \nomj\leq A_1 \aand A_2,\Delta$}
\AXC{$\ \ \ \ \ \ \ \ \ \ \ \ \ \ \ \,  \vdots$ \raisebox{1mm}{$\pi_3$}}
\noLine
\UIC{$\Gamma',\nomj\leq A_i \vd \Delta'$}
\LL{\fns $\aand_P$}
\UIC{$\Gamma', \nomj\leq A_1 \aand A_2 \vd \Delta'$}
\LL{\fns Cut$_{\,\nomj \leq A_1 \aand A_2}$}
\BIC{$\Gamma,\Gamma' \vd \Delta,\Delta'$}
\DP 

& 
\bottomAlignProof
$\rightsquigarrow$
& 

\bottomAlignProof
\AXC{\!\!\!\!\!\!\!\!\!\!\!$\vdots$ \raisebox{1mm}{$\pi_i$}}
\noLine
\UIC{$\Gamma \vd \nomj \leq A_i, \Delta$}
\AXC{$\ \ \ \ \ \ \ \ \ \ \ \ \ \ \ \, \vdots$ \raisebox{1mm}{$\pi_3$}}
\noLine
\UIC{$\Gamma',\nomj\leq A_i \vd \Delta'$}
\LL{\fns Cut$_{\,\nomj \leq A_i}$}
\BIC{$\Gamma,\Gamma' \vd \Delta,\Delta'$}
\DP 
 \\
\end{tabular}
\end{center}
 }

To show that $\mathbf{A.L}\Sigma$ satisfies all the other conditions is immediate and it requires inspecting all the rules in $\mathbf{A.L}\Sigma$.


\end{proof}

\section{Conclusions}
\label{sec:Conclusions}

In the present paper, we have showcased a methodology for introducing labelled calculi for nonclassical logics (LE-logics) in a uniform and principled way, taking the basic normal lattice-based modal logic $\mathbf{L}$ and some of its axiomatic extensions as a case study. This methodology hinges on the use of semantic information to generate calculi which are guaranteed by their design to enjoy a set of basic desirable  properties (soundness, syntactic completeness, conservativity, cut elimination and subformula property). Interestingly, the methodology showcased in the present paper naturally imports  the one developed and applied in \cite{GMPTZ} in the context of proper display calculi to the  proof-theoretic format of labelled calculi. Specifically, just like  the algorithm ALBA, the main tool in unified correspondence theory, was used in \cite{GMPTZ,ChnGrePalTzi21} to generate proper display calculi for basic (D)LE-logics in arbitrary signatures  and their axiomatic extensions defined by analytic inductive axioms,  in the present paper, ALBA is used to generate labelled calculi for $\mathbf{L}$ and 31 of its axiomatic extensions. Similarly to extant labelled calculi in the literature (viz.~those introduced in \cite{negri2005proof}), the language of the calculi introduced in the present paper manipulate a language which properly extends the original language of the logic, and includes {\em labels}. However, the language of  these labels is the same language manipulated by ALBA, the intended interpretation of which is provided by  a suitable {\em algebraic} environment, rather than by a relational one; specifically, by the canonical extensions of the algebras in the class canonically associated with the given logic. Just like the use of canonical extensions as a semantic environment for unified correspondence theory has allowed for the mechanization and  uniform generalization of  correspondence arguments from classical normal modal logic to the much wider setting of normal LE-logics without relying on the availability of any particular relational semantics, this same semantic setting allows for  the uniform generation of labelled calculi for LE-logics in a way that does not rely on a particular relational semantics. However, via general duality theoretic facts, these calculi will be sound also w.r.t.~any relational semantic environment for the given logic, and can also provide a ``blueprint'' for the introduction of labelled calculi designed to capture the logics of specific classes of relational structures (cf.~\cite{ICLA2022relational}). In future work, we will generalize the current results to arbitrary LE-signatures, and establish systematic connections, via formal translations, between proper display calculi and labelled calculi for LE-logics.

\bibliography{ref}

\begin{thebibliography}{10}

\bibitem{ChnGrePalTzi21}
J.~Chen, G.~Greco, A.~Palmigiano, and A.~Tzimoulis.
\newblock Syntactic completeness of proper display calculi.
\newblock {\em ACM Transactions on Computational Logic}, 2022.

\bibitem{conradie2021rough}
W.~Conradie, S.~Frittella, K.~Manoorkar, S.~Nazari, A.~Palmigiano,
  A.~Tzimoulis, and N.~M. Wijnberg.
\newblock Rough concepts.
\newblock {\em Information Sciences}, 561:371--413, 2021.

\bibitem{conradie2016categories}
W.~Conradie, S.~Frittella, A.~Palmigiano, M.~Piazzai, A.~Tzimoulis, and N.~M.
  Wijnberg.
\newblock Categories: how {I} learned to stop worrying and love two sorts.
\newblock In {\em WoLLIC}, pages 145--164. Springer, 2016.

\bibitem{conradie2017toward}
W.~Conradie, S.~Frittella, A.~Palmigiano, M.~Piazzai, A.~Tzimoulis, and N.~M.
  Wijnberg.
\newblock Toward an epistemic-logical theory of categorization.
\newblock {\em EPTCS}, 251, 2017.

\bibitem{CoGhPa14}
W.~Conradie, S.~Ghilardi, and A.~Palmigiano.
\newblock {Unified Correspondence}.
\newblock In {\em Johan van {B}enthem on {L}ogic and {I}nformation {D}ynamics},
  volume~5 of {\em Outstanding Contributions to Logic}, pages 933--975.
  Springer International Publishing, 2014.

\bibitem{CoPa12}
W.~Conradie and A.~Palmigiano.
\newblock Algorithmic correspondence and canonicity for distributive modal
  logic.
\newblock {\em Annals of Pure and Applied Logic}, 163(3):338 -- 376, 2012.

\bibitem{CP-nondist}
W.~Conradie and A.~Palmigiano.
\newblock {Algorithmic correspondence and canonicity for non-distributive
  logics}.
\newblock {\em Annals of Pure and Applied Logic}, 170(9):923--974, 2019.

\bibitem{conradie2020non}
W.~Conradie, A.~Palmigiano, C.~Robinson, and N.~Wijnberg.
\newblock Non-distributive logics: from semantics to meaning.
\newblock In A.~Rezus, editor, {\em Contemporary Logic and Computing}, volume~1
  of {\em Landscapes in Logic}, pages 38--86. College Publications, 2020.

\bibitem{LatticesOrder}
B.A. Davey and H.A. Priestley.
\newblock {\em {Introduction to lattices and order}}.
\newblock Cambridge University Press, New York, second edition, 2002.

\bibitem{MASSA}
A.~De~Domenico and G.~Greco.
\newblock Algorithmic correspondence and analytic rules.
\newblock In {\em Advances in Modal Logic}, volume~14, pages 371--389. College
  Publications, 2022.

\bibitem{DGP05}
J.M. Dunn, M.~Gehrke, and A.~Palmigiano.
\newblock Canonical extensions and relational completeness of some structural
  logics.
\newblock {\em The Journal of Symbolic Logic}, 70(3):713--740, 2005.

\bibitem{Multitype}
S.~Frittella, G.~Greco, A.~Kurz, A.~Palmigiano, and V.~Sikimi\'{c}.
\newblock Multi-type display calculus for dynamic epistemic logic.
\newblock {\em Journal of Logic and Computation}, 26(6):2017--2065, 2016.

\bibitem{ganter2012formal}
B.~Ganter and R.~Wille.
\newblock {\em Formal concept analysis: mathematical foundations}.
\newblock Springer Science \& Business Media, 2012.

\bibitem{GH01}
M.~Gehrke and J.~Harding.
\newblock Bounded lattice expansions.
\newblock {\em Journal of Algebra}, 238:345--371, 2001.

\bibitem{GMPTZ}
G.~Greco, M.~Ma, A.~Palmigiano, A.~Tzimoulis, and Z.~Zhao.
\newblock Unified correspondence as a proof-theoretic tool.
\newblock {\em Journal of Logic and Computation}, 28(7):1367--1442, 2016.

\bibitem{negri2005proof}
S.~Negri.
\newblock Proof analysis in modal logic.
\newblock {\em Journal of Philosophical Logic}, 34:507--544, 2005.

\bibitem{Negri2012}
S.~Negri and R.~Dyckhoff.
\newblock Proof analysis in intermediate logics.
\newblock {\em Archive for Mathematical Logic}, 51(1):71--92, 2012.

\bibitem{pawlak1982rough}
Z.~Pawlak.
\newblock Rough sets.
\newblock {\em IJCIS}, 11(5):341--356, 1982.

\bibitem{Sim94}
A.K. Simpson.
\newblock {\em {The Proof Theory and Semantics of Intuitionistic Modal Logic}}.
\newblock PhD dissertation, University of Edinburgh, 1994.

\bibitem{ICLA2022relational}
I.~van~der Berg, A.~De~Domenico, G.~Greco, K.~Manoorkar, A.~Palmigiano, and
  M.~Panettiere.
\newblock Labelled calculi for the logics of rough concepts.
\newblock {\em submitted}, 2023.

\bibitem{Wa98}
H.~Wansing.
\newblock {\em Displaying modal logic}.
\newblock Kluwer, 1998.

\bibitem{Wan02}
H.~Wansing.
\newblock {\em Sequent Systems for Modal Logics}, volume~8, pages 61--145.
\newblock Springer, Dordrecht, 2002.

\end{thebibliography}
\bibliographystyle{plain}

\begin{appendix}

\section{Proper labelled calculi}
\label{sec:ProperAlgebraicLabelledCalculi}

In what follows we provide a formal definition of the display property and proper labelled calculi.

\begin{definition}[Display]
\label{def:Display}
A nominal $\nomj$ (resp.~conominal $\cnomm$) is always {\em in display} in a labelled formula $\nomj\leq A$ (resp.~$A\leq \cnomm$), and is {\em in display} in a pure structure $t$ iff $t=\nomj \leq \NCT$ (resp.~$t=\NCT \leq \cnomm$) for some term $\NCT$ such that $\nomj$ does not occur in $\NCT$. 

A pure structure $\nomj \leq \NCT$ (resp.~$\NCT \leq \cnomm$) is {\em  in display} in a sequent $s$ if 
 $\nomj$ (resp.~$\cnomm$)  is in display  in each structure of $s$ in which it occurs.

A labelled formula $\nomj \leq A$ (resp.~$A \leq \cnomm$) is {\em  in display} in a sequent $s$ if  $\nomj$ (resp.~$\cnomm$) is in display  in each structure of $s$ in which it occurs.

A proof system enjoys the {\em display property} iff for every derivable sequent $s = \Gamma \vdash \Delta$ and every structure $\sigma \in s$, the sequent $s$ can be equivalently transformed, using the rules of the system, into a sequent $s'$ s.t.~$\sigma$ occurs in display in $s'$ (in this case we might say that $\sigma$ is {\em displayable}).\footnote{Notice that we are not requiring that every meta-structure is displayable.}
\end{definition}
\begin{definition}[Proper labelled calculi]
\label{def:ProperLabelledCalculi}
A proof system is a {\em proper labelled calculus} if it satisfies the following list of conditions:



\noindent \textbf{C$_1$: Preservation of formulas.\;} Each formula occurring in a premise of an inference $r$ is a subformula of some formula in the conclusion of $r$.

\noindent \textbf{C$_2$: Shape-alikeness of parameters and formulas/terms in congruent structures.\;} (i) Congruent parameters are occurrences of the same (meta-)structure (i.e.~instantiations of structure metavariables in the application of a rule $R$ except for switch rules); (ii) instantiations of labelled formulas in the application of switch rules (in the case of $\mathbf{A.L}\Sigma$, occurrences of the form $\nomj \leq C$ and $C \leq \cnomm$) are congruent and the formulas in these occurrences instantiate the same formula metavariable (namely $C$); (iii) instantiations of pure structures in the application of switch rules (in the case of $\mathbf{A.L}\Sigma$, occurrences of the form $\nomj\leq \NCT''$ and $\NCT'' \leq \cnomm$) are congruent and exactly one term in these occurrences instantiate the same term metavariable (namely $\NCT''$).


\noindent \textbf{C$_3$: Non-proliferation of parameters and congruent structures.\;} (i) Each parameter in an inference $r$ is congruent to at most one constituent in the conclusion of $r$. (ii) Each nonparametric structure in the instantiation $r$ of a switch rule $R$ is congruent to at most one nonparametric structure in the conclusion of $r$. 

\noindent \textbf{C$_4$: Position-alikeness of parameters and congruent structures.\;} Congruent parameters and congruent structures occur in the same position (i.e.~either in  precedent position or in  succedent position) in their respective sequents.

\noindent \textbf{C$_5$: Display of principal constituents.\;} If a labelled formula $a$ is principal in the conclusion of an inference $r$, then $a$ is in display. 






\noindent \textbf{C$'_5$: Display-invariance of axioms and structural rules.\;} If a structure $\sigma$ occurs in the conclusion $s$ of a structural rule \ $\AXC{$s_1, \ldots, s_n$}\RL{\fns $R$}\UIC{$s$}\DP$ \ (where $R$ is an axiom whenever the set of premises is empty), then either $\sigma$ occurs in display in $s$, or a structure $\sigma'$ and a sequent $s'$ exist s.t.~$\sigma'$ is in display in $s'$, and $s'$ is derivable from $s$ via application of switch and adjunction rules only, and $\sigma$ and $\sigma'$ are congruent in this derivation. Moreover, if the rule $R$ is an axiom, then \ \AXC{ }\RL{\fns $R'$}\UIC{$s'$}\DP \ is an axiom of the calculus as well.

\noindent \textbf{C$_6$: Closure under substitution for succedent parts.\;} Each rule is closed under simultaneous substitution of (sets of) arbitrary structures for congruent labelled formulas occurring in succedent position.

\noindent \textbf{C$_7$: Closure under substitution for precedent parts.\;} Each rule is closed under simultaneous substitution of (sets of) arbitrary structures for congruent labelled formulas occurring in precedent position.

\noindent where $\osigma$ is a multi-set of structures and $\osigma \slash a$ means that $\osigma$ are substituted for $a$.

\noindent This condition caters for the step in the cut-elimination procedure in which the cut needs to be ``pushed up'' over rules in which the cut-formula in succedent position  is parametric.

\noindent \textbf{C$_8$: Eliminability of matching principal constituents.\;}
This condition requests a standard Gentzen-style checking, which is now limited to the case in which  both cut formulas are {\em principal}, i.e.~each of them has been introduced with the last rule application of each corresponding subdeduction. In this case, analogously to the proof Gentzen-style, condition C$_8$ requires being able to transform the given deduction into a deduction with the same conclusion in which either the cut is eliminated altogether, or is transformed in one or more applications of the cut rule(s), involving proper subformulas of the original cut-formula.

\end{definition}

We now provide the proof of Theorem \ref{thm:MetaCutElimination} stated in Section \ref{CutEliminationAndSubformulaProperty}.

\begin{proof} 
The proof is close to the proofs in \cite{Multitype} and \cite[Section 3.3, Appendix A]{Wan02}. As to the principal move (i.e.~both labelled cut formulas are principal), condition C$_8$ guarantees that this cut application can be eliminated.  
As to the parametric moves (i.e.~at least one labelled cut formula is parametric), we are in the following situation:

{\footnotesize{
\begin{center}
\AXC{\ \ \ $\vdots$ \raisebox{1mm}{$\pi_1$}}
\noLine
\def\fCenter{\vdash}
\UI$(\Pi \fCenter \Sigma)[a]^{suc}$
\AXC{\ \ \ \ \ $\vdots$ \raisebox{1mm}{$\pi_{2.1}$}}
\noLine
\UI$(\Gamma_1 \fCenter \Delta_1) [a_{u_1}]^{pre}$
\AXC{$\cdots{\phantom{\vdash}}$}
\AXC{\ \ \ \ \ $\vdots$ \raisebox{1mm}{$\pi_{2.n}$}}
\noLine
\UI$(\Gamma_n \fCenter \Delta_n) [a_{u_n}]^{pre}$
\noLine
\TIC{\ \,$\ddots\vdots\iddots$ \raisebox{0.3mm}{$\pi_2$}}
\noLine
\UI$(\Gamma \fCenter \Delta) [a]^{pre}$
\RL{\fns Cut}
\BI$\Pi, \Gamma \fCenter \Sigma, \Delta$
\DP
\end{center}
}}

\noindent where we assume that the cut labelled formula $a$ is parametric in the conclusion of $\pi_2$ (the other case is symmetric), and $(\Gamma \vdash \Delta)[a]^{pre}$ (resp.~$(\Pi \vdash \Sigma)[a]^{suc}$) means that $a$ occur in precedent (resp.~succedent) position in $\Gamma \vdash \Delta$ (resp.~$\Pi \vdash \Sigma$). 

Conditions C$_2$-C$_4$ make it possible to follow the history of that occurrence of $a$, since these conditions enforce that the history takes the shape of a tree, of which we consider each leaf. Let $a_{u_i}$ (abbreviated to $a_u$ from now on) be one such uppermost-occurrence in the history-tree of the parametric cut term $a$ occurring in $\pi_2$, and let $\pi_{2.i}$ be the subderivation ending in the sequent $\Gamma_i \vdash \Delta_i$, in which $a_u$ is introduced. 

Wansing's parametric case (1) splits into two subcases: (1a) $a_u$ is introduced in display; (1b) $a_u$ is not introduced in display. Condition C$'_5$ guarantees that (1b) can only be the case when $a_u$ has been introduced via an axiom. 

If (1a), then we can perform the following transformation:

{\footnotesize{
\begin{center}
\begin{tabular}{lcr}
\bottomAlignProof
\AXC{\ \ \ \,$\vdots$ \raisebox{1mm}{$\pi_1$}}
\noLine
\UI$(\Pi \fCenter \Sigma)[a]^{suc}$
\AXC{\ \ \ \ \,$\vdots$ \raisebox{1mm}{$\pi_{2.i}$}}
\noLine
\UI$(\Gamma_i \fCenter \Delta_i)[a_u]^{pre}$
\noLine
\UIC{\ \ \ \,$\vdots$ \raisebox{1mm}{$\pi_2$}}
\noLine
\UI$(\Gamma \fCenter \Delta) [a]^{pre}$
\RL{\fns Cut}
\BI$\Pi, \Gamma \fCenter \Sigma, \Delta$
\DisplayProof

 & \ \ \ $\rightsquigarrow$ \!\!\!\!\!&

\bottomAlignProof
\AXC{\ \ \ $\vdots$ \raisebox{1mm}{$\pi_1$}}
\noLine
\UI$(\Pi \fCenter \Sigma)[a]^{suc}$
\AXC{\ \ \ \ \ $\vdots$ \raisebox{1mm}{$\pi_{2.i}$}}
\noLine
\UI$(\Gamma_i \fCenter \Delta_i)[a_u]^{pre}$
\RL{\fns Cut'}
\BI$\Pi, \Gamma_i \fCenter \Sigma, \Delta_i$
\noLine
\UIC{\ \ \ \ \ \ \ \ \ \ \ \ \ \ \ \ \ \ \ \ \,$\vdots$ \raisebox{1mm}{$\pi_2 \,[\{\Pi, \Sigma\} \slash a]$}}
\noLine
\UI$\Pi, \Gamma \fCenter \Sigma, \Delta$
\DisplayProof
 \\
\end{tabular}
\end{center}
}}

\noindent where $\pi_2 \,[\{\Pi, \Sigma\} \slash a]$ is the derivation obtained by $\pi_2$ by substituting $\Pi, \Sigma$ for $a$ in $\pi_2$.\footnote{Notice that the writing $\pi_2 \,[\{\Pi, \Sigma\} \slash a]$ does not mean that $\Pi$ and $\Sigma$ remain untouched in $\pi_2$, namely it does not mean that every sequent in $\pi_2$ is of the form $\Pi, \Gamma'' \vdash \Sigma, \Delta''$ for some $\Gamma'', \Delta''$. Indeed, structures in $\Pi$ or in $\Sigma$ might play the role of active structures in some applications of switch rules occurring in $\pi_2$, if any.} Notice that the assumption that $a$ is parametric in the conclusion of $\pi_2$ and that $a_u$ is principal implies that $\pi_2$ has more than one node, and hence the transformation above results in a cut application of strictly lower height. Moreover, condition C$_7$ implies that Cut' is well defined and the substitution of $\{\Pi, \Sigma\}$ for $a$ in $\pi_2$ gives rise to an admissible derivation $\pi_2 \,[\{\Pi, \Sigma\}/a]^{pre}$ in the calculus (use C$_6$ for the symmetric case).
If (1b), i.e.\ if $a_u$ is the principal formula of an axiom, the situation is illustrated below in the derivation on the left-hand side:

{\footnotesize{
\begin{center}
\begin{tabular}{lcr}
\bottomAlignProof
\AXC{\ \ \ $\vdots$ \raisebox{1mm}{$\pi_1$}}
\noLine
\UIC{$(\Pi \fCenter \Sigma)[a]^{suc}$}
%
\AXC{$(\Gamma_i \fCenter \Delta_i) [a_u]^{pre}$}
\noLine
\UIC{\ \ \ $\vdots$ \raisebox{1mm}{$\pi_2$}}
\noLine
\UI$(\Gamma \fCenter \Delta) [a]^{pre}$
\RL{\fns Cut}
\BI$(\Pi, \Gamma \fCenter \Sigma, \Delta$
\DP

 & \ \ \ \ $\rightsquigarrow$ \!\!\!\!\!\!&

\bottomAlignProof
\AXC{\ \ \,$\vdots$ \raisebox{1mm}{$\pi_1$}}
\noLine
\UIC{$(\Pi \fCenter \Sigma)[a]^{suc}$}
%
\AXC{$(\Gamma'_i \fCenter \Delta'_i) [a_u]^{pre}$}
\RL{\fns Cut'}
\BI$\Pi, \Gamma' \fCenter \Sigma, \Delta'$
\noLine
\UIC{\ \ \ $\vdots$ \raisebox{1mm}{$\pi'$}}
\noLine
\UI$(\Gamma_i \fCenter \Delta_i) [\{\Pi, \Sigma\}/a]$
\noLine
\UIC{\ \ \ \ \ \ \ \ \ \ \ \ \ \ \ \ \ \ \ \, $\vdots$ \raisebox{1mm}{$\pi_2 \,[\{\Pi,\Sigma\}/a]$}}
\noLine
\UI$\Pi, \Gamma \fCenter \Sigma, \Delta$
\DP
 \\
\end{tabular}
\end{center}
}}

\noindent where $(\Gamma_i \fCenter \Delta_i) [a_u]^{pre}[a]^{suc}$ is an axiom. Then, condition C$'_5$ implies that some sequent $(\Gamma'_i \fCenter \Delta'_i) [a_u]^{pre}[a]^{suc}$ exists, which is display-equivalent to the first axiom, and in which $a_u$ occurs in display. This new sequent can be either identical to $(\Gamma_i \fCenter \Delta_i) [a_u]^{pre}[a]^{suc}$, in which case we proceed as in case (1a), or it can be different, in which case, condition C$'_5$ guarantees that it  is an axiom as well. Further, if $\pi$ is the derivation consisting of applications of adjunction and switch rules which transform the latter axiom into the former, then let $\pi' = \pi \,[\{\Pi,\Sigma\}/a_u]$. As discussed when treating (1a), the assumptions imply that $\pi_2$ has more than one node, so the transformation described above results in a cut application of strictly lower height. 
 Moreover, condition C$_7$ implies that Cut' is well defined and substituting $\{\Pi, \Sigma\}$ for $a_u$ in $\pi_2$ and in $\pi$ gives rise to admissible derivations $\pi_2 \,[\{\Pi, \Sigma\}/a_u]$ and $\pi'$ in the calculus (use C$_6$ for the symmetric case).

As to Wansing's case (2), assume that $a_u$ has been introduced as a parameter in the conclusion of $\pi_{2.i}$ by an application $r$ of the rule $R$. 

Therefore, the transformation below yields a derivation where $\pi_1$ is not used at all and the cut is not applied.

{\footnotesize{
\begin{center}
\begin{tabular}{rcl}
\bottomAlignProof
\AXC{\ \ \ \,$\vdots$ \raisebox{1mm}{$\pi_1$}}
\noLine
\UI$(\Pi \fCenter \Sigma)[a]^{suc}$
\AXC{\ \ \ \ \ \,$\vdots$ \raisebox{1mm}{$\pi_{2.i}$}}
\noLine
\UI$(\Gamma_i \fCenter \Delta_i) [a_u]^{pre}$
\noLine
\UIC{\ \ \ \ $\vdots$ \raisebox{1mm}{$\pi_2$}}
\noLine
\UI$(\Gamma \fCenter \Delta) [a]^{pre}$
\RL{\fns Cut}
\BI$\Pi, \Gamma \fCenter \Sigma, \Delta$
\DisplayProof

 & \ \ \ \ $\rightsquigarrow$ \!\!\!\!\!\!\!\!\!\!\!\!\!\!&

\bottomAlignProof
\AXC{\ \ \ \ \ $\vdots$ \raisebox{1mm}{$\pi_{2.i}'$}}
\noLine
\UI$(\Gamma_i \fCenter \Delta_i)[\{\Pi, \Sigma\}/a_u]$
\noLine
\UIC{\ \ \ \ \ \ \ \ \ \ \ \ \ \ \ \ \ \ \ \, $\vdots$ \raisebox{1mm}{$\pi_2 \,[\{\Pi, \Sigma\}/a]$}}
\noLine
\UI$\Pi, \Gamma \fCenter \Sigma, \Delta$
\DP
\end{tabular}
\end{center}
}}

From this point on, the proof proceeds like in \cite{Wan02}. 

\end{proof}




\end{appendix}

\end{document}